\documentclass[12pt,twoside]{amsart}
\usepackage{amsmath, amsthm, amscd, amsfonts, amssymb, graphicx}
\usepackage[bookmarksnumbered, plainpages]{hyperref}

\newtheorem{thm}{Theorem}[section]
\newtheorem{cor}[thm]{Corollary}

\newtheorem{prop}[thm]{Proposition}

\numberwithin{equation}{section}

\begin{document}

\title{\bf Affine connections of non integrable distributions}
\author{Yong Wang}

\thanks{{\scriptsize
\hskip -0.4 true cm \textit{2010 Mathematics Subject Classification:}
53C40; 53C42.
}}

\maketitle

\begin{abstract}
 In this paper, we study non integrable distributions in a Riemannian manifold with a semi-symmetric metric connection, a semi-symmetric non-metric connection and a statistical connection. We
obtain the Gauss, Codazzi, and Ricci equations for non integrable distributions with respect to the semi-symmetric metric connection, the semi-symmetric non-metric connection and the statistical connection. As applications, we obtain Chen's inequalities for non integrable distributions of
real space forms endowed with a semi-symmetric metric connection and a semi-symmetric non-metric connection. We give some examples of non integrable distributions in a Riemannian manifold with affine connections.
We find some new examples of Einstein distributions and distributions with constant scalar curvature.
\end{abstract}

\vskip 0.2 true cm

%------------------------------------------------------------------------------------%

\pagestyle{myheadings}
\markboth{\rightline {\scriptsize Wang}}
         {\leftline{\scriptsize Affine connections of non integrable distributions}}

\bigskip
\bigskip

%------------------------------------------------------------------------------------%
%------------------------------------------------------------------------------------%

\section{ Introduction}

H. A. Hayden introduced the notion of a semi-symmetric metric connection on a
Riemannian manifold \cite{HA}. K. Yano studied a Riemannian manifold endowed with
a semi-symmetric metric connection \cite{Ya}. Some properties of a Riemannian manifold
and a hypersurface of a Riemannian manifold with a semi-symmetric metric
connection were studied by T. Imai \cite{I1,I2}. Z. Nakao \cite{NA} studied submanifolds of
a Riemannian manifold with semi-symmetric metric connections. N. S. Agashe and
M. R. Chafle introduced the notion of a semisymmetric non-metric connection and
studied some of its properties and submanifolds of a Riemannian manifold with a
semi-symmetric non-metric connection \cite{AC1,AC2}. In \cite{Vo}, Vos studied submanifolds of statistical manifolds and got the
the Gauss, Codazzi, and Ricci equations for statistical submanifolds.
In \cite{Mu}, the author considered non integrable distributions in a Riemannian manifold. The second fundamental form was defined and the Gauss equation for non integrable distributions
was established. In this paper, We establish Gauss formulas and Weingarten formulas and
obtain the Gauss, Codazzi, and Ricci equations for non integrable distributions with respect to a semi-symmetric metric connection, a semi-symmetric non-metric connection and a statistical connection.\\

On the other hand, one of the basic problems in submanifold theory is to find
simple relationships between the extrinsic and intrinsic invariants of a submanifold.
B. Y. Chen \cite{BYC0,BYC1,BYC2} established inequalities in this respect, called Chen inequalities.
Afterwards, many geometers studied similar problems for different submanifolds in
various ambient spaces.
 In \cite{MO1,OM}, Mihai and $\ddot{{\rm O}}$zg$\ddot{{\rm u}}$r studied Chen inequalities for submanifolds of real
space forms with a semi-symmetric metric connection and a semi-symmetric non-metric connection, respectively.
In this paper, we obtain Chen's inequalities for non integrable distributions of
real space forms endowed with a semi-symmetric metric connection and a semi-symmetric non-metric connection.\\
\indent It is a interesting problem to find Einstein manifolds and manifolds with constant scalar curvature. In \cite{DU}, using warped product spaces, Dobarro and Unal found
 Einstein manifolds and manifolds with constant scalar curvature. In \cite{SO,Wa1,Wa2}, authors found
Einstein manifolds and manifolds with constant scalar curvature with a semi-symmetric metric connection and a semi-symmetric non-metric connection.
In this paper, we define Einstein distributions and distributions with constant scalar curvature. We define warped product distributions on $\mathbb{R}\times \mathbb{S}^3$ and $\mathbb{R}\times H_3$ where $\mathbb{S}^3$ and $H_3$ are
the $3$-dimensional sphere and the $3$-dimensional Heisenberg group respectively. We find some new examples of Einstein distributions and distributions with constant scalar curvature.\\
\indent In Section 2, we establish the Gauss formula and the Weingarten formula and
obtain the Gauss, Codazzi, and Ricci equations for non integrable distributions with respect to a semi-symmetric metric connection. In this case, the Chen inequality is proved.
In Section 3, we establish the Gauss formula and the Weingarten formula and
obtain the Gauss, Codazzi, and Ricci equations for non integrable distributions with respect to a semi-symmetric non-metric connection. We also prove the Chen inequality in this case.
In Section 4, we establish the Gauss formula and the Weingarten formula and
obtain the Gauss, Codazzi, and Ricci equations for non integrable distributions with respect to a statistical connection.
In Section 5, We give some examples of non integrable distributions in a Riemannian manifold with affine connections.
We find some new examples of Einstein distributions and distributions with constant scalar curvature.

%------------------------------------------------------------------------------------%

\vskip 1 true cm

\section{ Non integrable distributions with a semi-symmetric metric connection}

Let $(M,g)$ be a smooth Riemannian manifold, ${\rm dim}M=m$, and $\nabla$ be the Levi-Civita connection associated to the Riemannian metric $g$. We denote $\Gamma(M)$ the $C^{\infty}(M)$-module of vector
fields on $M$ and by $\nabla_XY$ the covariant derivative of $Y$ with respect to $X$ if $X,Y\in \Gamma(M)$. Let $D\subseteq TM$ be a non integrable distribution; that is, a subbundle of the tangent bundle $TM$ with constant
rank $n$
and there exist $X,Y \in\Gamma(D)$ such that $[X,Y]$ is not in $\Gamma(D)$ where $\Gamma(D)$ is the space of sections of $D$. The distribution $D$ inherits a metric tensor field $g^D$ from the original $g$ in $M$. Let $D^\bot\subseteq
TM$ is the orthogonal distribution to $D$ which inherits a metric tensor field $g^{D^\bot}$ from the $g$ and then $g=g^D\oplus g^{D^\bot}$. Let $\pi^D:TM\rightarrow D$, $\pi^{D^\bot}:TM\rightarrow D^\bot$ be the projections.
For $X,Y \in\Gamma(D)$, we define $\nabla^D_XY=\pi^D(\nabla_XY)$ and $[X,Y]^D=\pi^D([X,Y])$ and $[X,Y]^{D^\bot}=\pi^{D^\bot}([X,Y])$. By \cite{Mu}, we have for $X,Y \in\Gamma(D)$ and $f\in C^{\infty}(M)$
\begin{equation}
\nabla^D_{fX}Y=f\nabla^D_{X}Y,~~\nabla^D_{X}(fY)=X(f)Y+f\nabla^D_{X}Y,
\end{equation}
\begin{equation}
\nabla^D_Xg^D=0,~~~~T(X,Y):=\nabla^D_{X}Y-\nabla^D_{Y}X-[X,Y]=-[X,Y]^{D^\bot},
\end{equation}
and
\begin{equation}
\nabla_XY=\nabla^D_{X}Y+B(X,Y),~~B(X,Y)=\pi^{D^\bot}\nabla_XY.
\end{equation}
We note that $B(X,Y)\neq B(Y,X)$.\\
\indent Let $U\in \Gamma(TM)$ be a vector field and $\omega$ be a $1$-form defined by $\omega(V)=g(U,V)$ for any $V\in \Gamma(TM)$. We define the semi-symmetric metric connection on $M$
\begin{equation}
\widetilde{\nabla}_XY=\nabla_{X}Y+\omega(Y)X-g(X,Y)U.
\end{equation}
Let $U^D=\pi^DU$ and $U^{D^\bot}=\pi^{D^\bot}U$, then $U=U^D+U^{D^\bot}$. Let
\begin{equation}
\widetilde{\nabla}_XY=\widetilde{\nabla}^D_{X}Y+\widetilde{B}(X,Y),~~\widetilde{\nabla}^D_{X}Y=\pi^D\widetilde{\nabla}_{X}Y,~~\widetilde{B}(X,Y)=\pi^{D^\bot}\widetilde{\nabla}_XY.
\end{equation}
We call the $\widetilde{B}(X,Y)$ as the second fundamental form with respect to the semi-symmetric metric connection. By (2.3)-(2.5), we have
\begin{equation}
\widetilde{\nabla}^D_{X}Y={\nabla}^D_{X}Y+\omega(Y)X-g(X,Y)U^D,~~\widetilde{B}(X,Y)={B}(X,Y)-g(X,Y)U^{D^\bot}.
\end{equation}
By (2.2) and (2.6), we have

\begin{equation}
\widetilde{\nabla}^D_{X}(g^D)=0,~~\widetilde{T}^D(X,Y)=-[X,Y]^{D^\bot}+\omega(Y)X-\omega(X)Y.
\end{equation}
Similarly to the case $D=TM$, we have
\begin{thm}
There exists a unique linear connection $\widetilde{\nabla}^D: \Gamma(D)\times \Gamma(D)\rightarrow \Gamma(D)$ on $D$, which satisfies the property (2.7).
\end{thm}
Let $\{E_1,\cdots,E_n\}$ be the orthonormal basis on $D$. We define the mean curvature vector associated to $\widetilde{\nabla}$ on $D$ by $\widetilde{H}=\frac{1}{n}\sum_{i=1}^n\widetilde{B}(E_i,E_i)\in \Gamma(D^\bot).$
By (2.6), then $\widetilde{H}=H-U^{D^\bot},$ where ${H}=\frac{1}{n}\sum_{i=1}^n{B}(E_i,E_i).$ If $\widetilde{H}=0$, we say that $D$ is minimal with respect to the semi-symmetric metric connection $\widetilde{\nabla}$.
We say that $D$ is totally geodesic with respect to the semi-symmetric metric connection $\widetilde{\nabla}$ if $\widetilde{B}(X,Y)+\widetilde{B}(Y,X)=0$
\begin{prop}
1) If $D$ is totally geodesic with respect to the connection ${\nabla}$, then $D$ is totally geodesic with respect to the semi-symmetric metric connection $\widetilde{\nabla}$ if and only if $U\in\Gamma(D).$\\
2) $H=\widetilde{H}$ if and only if $U\in\Gamma(D).$
\end{prop}
Let $h(X,Y)=\frac{1}{2}[B(X,Y)+B(Y,X)]$ and $\widetilde{h}(X,Y)=\frac{1}{2}[\widetilde{B}(X,Y)+\widetilde{B}(Y,X)].$ If $h=Hg^D$ (resp. $\widetilde{h}=\widetilde{H}g^D$), we say that $D$ is umbilical with respect to
$\nabla$ (resp. $\widetilde{\nabla}$). By (2.6), we have
\begin{prop}
$D$ is umbilical with respect to
$\nabla$ if and only if  $D$ is umbilical with respect to
$\widetilde{\nabla}$.
\end{prop}
Let $\xi\in\Gamma(D^{\bot})$ and $X\in\Gamma(D)$, then by (2.4), we have
\begin{equation}
\widetilde{\nabla}_X\xi={\nabla}_X\xi+\omega(\xi)X.
\end{equation}
Let $A_\xi:\Gamma(D)\rightarrow \Gamma(D)$ be the shape operator with respect to $\nabla$ defined by
\begin{equation}
g^D(A_\xi X,Y):=g^{D^\bot}(B(X,Y),\xi).
\end{equation}
Let $\nabla_X\xi=\pi^D\nabla_X\xi+L^{\bot}_X\xi$, then
\begin{equation}
\pi^D\nabla_X\xi=-A_\xi X,~~\nabla_X\xi=-A_\xi X+L^{\bot}_X\xi,
\end{equation}
which we called the Weingarten formula with respect to $\nabla$ and $L^{\bot}_X\xi:\Gamma(D)\times \Gamma(D^{\bot})\rightarrow \Gamma(D^{\bot})$ is a metric connection on $D^{\bot}$ along $\Gamma(D)$.
Let $\widetilde{A}_{\xi}=(A_{\xi}-\omega(\xi))I$, then by (2.8) and (2.10), we have
\begin{equation}
\widetilde{\nabla}_X\xi=-\widetilde{A}_\xi X+L^{\bot}_X\xi,
\end{equation}
which we called the Weingarten formula with respect to $\widetilde{\nabla}$.\\
\indent Given $X_1,X_2,X_3\in\Gamma(TM)$, the curvature tensor $\widetilde{R}$ with respect to $\widetilde{\nabla}$ is defined by
\begin{equation}
\widetilde{R}(X_1,X_2)X_3:=\widetilde{\nabla}_{X_1}\widetilde{\nabla}_{X_2}X_3-\widetilde{\nabla}_{X_2}\widetilde{\nabla}_{X_1}X_3-\widetilde{\nabla}_{[X_1,X_2]}X_3.
\end{equation}
Given $X_1,X_2,X_3\in\Gamma(D)$, the curvature tensor $\widetilde{R}^D$ on $D$ with respect to $\widetilde{\nabla}^D$ is defined by
\begin{equation}
\widetilde{R}^D(X_1,X_2)X_3:=\widetilde{\nabla}^D_{X_1}\widetilde{\nabla}^D_{X_2}X_3-\widetilde{\nabla}^D_{X_2}\widetilde{\nabla}^D_{X_1}X_3-\widetilde{\nabla}^D_{[X_1,X_2]^D}X_3-\pi^D[[X_1,X_2]^{D^{\bot}},X_3].
\end{equation}
In (2.13), $\widetilde{R}^D$ is a tensor field by adding the extra term $-\pi^D[[X_1,X_2]^{D^{\bot}},X_3]$.
Given $X_1,X_2,X_3,X_4\in\Gamma(D)$, the Riemannian curvature tensor $\widetilde{R}$, $\widetilde{R}^D$ are defined by
\begin{equation}
\widetilde{R}(X_1,X_2,X_3,X_4)=g(\widetilde{R}(X_1,X_2)X_3,X_4),~~\widetilde{R}^D(X_1,X_2,X_3,X_4)=g(\widetilde{R}^D(X_1,X_2)X_3,X_4).
\end{equation}
\begin{thm}
Given $X,Y,Z,W\in\Gamma(D)$, we have
\begin{align}
\widetilde{R}(X,Y,Z,W)&=\widetilde{R}^D(X,Y,Z,W)-g(B(X,W),B(Y,Z))+g(B(Y,W),B(X,Z))\\
&+g(Y,Z)\omega(B(X,W))-g(X,Z)\omega(B(Y,W))+g(X,W)\omega(B(Y,Z))\notag\\
&-g(Y,W)\omega(B(X,Z))-g(Y,Z)g(X,W)\omega(U^{D^\bot})\notag\\
&+g(X,Z)g(Y,W)\omega(U^{D^\bot})+g(B(Z,W),[X,Y]).\notag
\end{align}
Here Equation (2.15) is called the Gauss equation for $D$ with respect to $\widetilde{\nabla}$.
\end{thm}

\begin{proof}
From Equations (2.5) and (2.11), we have for $X,Y,Z\in\Gamma(D)$
\begin{align}
\widetilde{\nabla}_X\widetilde{\nabla}_YZ&=\widetilde{\nabla}^D_X\widetilde{\nabla}^D_YZ+\widetilde{B}(X,\widetilde{\nabla}^D_YZ)\\
&-A_{\widetilde{B}(Y,Z)}X+\omega(\widetilde{B}(Y,Z))X+L^{\bot}_X(\widetilde{B}(Y,Z))\notag,
\end{align}
\begin{align}
\widetilde{\nabla}_Y\widetilde{\nabla}_XZ&=\widetilde{\nabla}^D_Y\widetilde{\nabla}^D_XZ+\widetilde{B}(Y,\widetilde{\nabla}^D_XZ)\\
&-A_{\widetilde{B}(X,Z)}Y+\omega(\widetilde{B}(X,Z))Y+L^{\bot}_Y(\widetilde{B}(X,Z))\notag,
\end{align}
By (2.11) and for $X_1,X_2\in\Gamma(TM)$,
\begin{align}
\widetilde{\nabla}_{X_1}X_2=\widetilde{\nabla}_{X_2}X_1+[X_1,X_2]+\omega(X_2)X_1-\omega(X_1)X_2,
\end{align}
we get
\begin{align}
\widetilde{\nabla}_{[X,Y]^{D^{\bot}}}Z=-A_{[X,Y]^{D^{\bot}}}Z+L^{\bot}_Z([X,Y]^{D^{\bot}})+\omega(Z){[X,Y]^{D^{\bot}}}+[{[X,Y]^{D^{\bot}}},Z].
\end{align}
By $\widetilde{\nabla}_{[X,Y]}Z=\widetilde{\nabla}_{[X,Y]^{D}}Z+\widetilde{\nabla}_{[X,Y]^{D^{\bot}}}Z$ and (2.19) and (2.5), we have
\begin{align}
\widetilde{\nabla}_{[X,Y]}Z&=\widetilde{\nabla}^D_{[X,Y]^{D}}Z+\widetilde{B}([X,Y]^{D},Z)
-A_{[X,Y]^{D^{\bot}}}Z\\
&+L^{\bot}_Z([X,Y]^{D^{\bot}})+\omega(Z){[X,Y]^{D^{\bot}}}+[{[X,Y]^{D^{\bot}}},Z].\notag
\end{align}
By (2.12),(2.13),(2.16),(2.17) and (2.20), we have
\begin{align}
\widetilde{R}(X,Y)Z=&\widetilde{R}^D(X,Y)Z-\pi^{D^{\bot}}[{[X,Y]^{D^{\bot}}},Z]+\widetilde{B}(X,\widetilde{\nabla}^D_YZ)\\
&-\widetilde{B}(Y,\widetilde{\nabla}^D_XZ)
-\widetilde{B}([X,Y]^{D},Z)-A_{\widetilde{B}(Y,Z)}X+A_{\widetilde{B}(X,Z)}Y\notag\\
&+L^{\bot}_X(\widetilde{B}(Y,Z))-L^{\bot}_Y(\widetilde{B}(X,Z))
+\omega(\widetilde{B}(Y,Z))X-\omega(\widetilde{B}(X,Z))Y\notag\\
&+A_{[X,Y]^{D^{\bot}}}Z
-L^{\bot}_Z([X,Y]^{D^{\bot}})-\omega(Z){[X,Y]^{D^{\bot}}}.\notag
\end{align}
By (2.9) and (2.21), we get (2.15).
\end{proof}

\begin{cor} If $U=0$, then $\omega=0$ and $\widetilde{\nabla}=\nabla$, and we have
\begin{align}
{R}(X,Y,Z,W)=&{R}^D(X,Y,Z,W)-g(B(X,W),B(Y,Z))\\
&+g(B(Y,W),B(X,Z))+g(B(Z,W),[X,Y]).\notag
\end{align}
\end{cor}

\begin{thm}
Given $X,Y,Z\in\Gamma(D)$, we have
\begin{align}
(\widetilde{R}(X,Y)Z)^{D^{\bot}}=&(L^{\bot}_X\widetilde{B})(Y,Z)-(L^{\bot}_Y\widetilde{B})(X,Z)\\
&-\omega(X)\widetilde{B}(Y,Z)+\omega(Y)\widetilde{B}(X,Z)-\pi^{D^{\bot}}[{[X,Y]^{D^{\bot}}},Z]\notag\\
&-L^{\bot}_Z([X,Y]^{D^{\bot}})-\omega(Z){[X,Y]^{D^{\bot}}},\notag
\end{align}
where $(L^{\bot}_X\widetilde{B})(Y,Z)=L^{\bot}_X(\widetilde{B}(Y,Z))-\widetilde{B}(\widetilde{\nabla}^D_XY,Z)-\widetilde{B}(Y,\widetilde{\nabla}^D_XZ).$ Equation (2.23) is called the Codazzi
equation with respect to $\widetilde{\nabla}$.
\end{thm}
\begin{proof}
From (2.21), we have
\begin{align}
(\widetilde{R}(X,Y)Z)^{D^{\bot}}=&-\pi^{D^{\bot}}[{[X,Y]^{D^{\bot}}},Z]+\widetilde{B}(X,\widetilde{\nabla}^D_YZ)\\
&-\widetilde{B}(Y,\widetilde{\nabla}^D_XZ)
-\widetilde{B}([X,Y]^{D},Z)
+L^{\bot}_X(\widetilde{B}(Y,Z))\notag\\
&-L^{\bot}_Y(\widetilde{B}(X,Z))
-L^{\bot}_Z([X,Y]^{D^{\bot}})-\omega(Z){[X,Y]^{D^{\bot}}}.\notag
\end{align}
By $[X,Y]^D=\widetilde{\nabla}^D_XY-\widetilde{\nabla}^D_YX-\omega(Y)X+\omega(X)Y$ and the definition of $(L^{\bot}_X\widetilde{B})(Y,Z)$ and (2.24), we get (2.23).
\end{proof}

\begin{cor} If $U=0$, then we have
\begin{align}
({R}(X,Y)Z)^{D^{\bot}}=&(L^{\bot}_X{B})(Y,Z)-(L^{\bot}_Y{B})(X,Z)\\
&-\pi^{D^{\bot}}[{[X,Y]^{D^{\bot}}},Z]
-L^{\bot}_Z([X,Y]^{D^{\bot}}).\notag
\end{align}
\end{cor}

\begin{thm}
Given $X,Y\in\Gamma(D)$, $\xi\in\Gamma(D^{\bot})$, we have
\begin{align}
(\widetilde{R}(X,Y)\xi)^{D^{\bot}}=-\widetilde{B}(X,\widetilde{A}_{\xi}Y)+\widetilde{B}(Y,\widetilde{A}_{\xi}X)+\widetilde{R}^{L^{\bot}}(X,Y)\xi
\end{align}
where
\begin{align}
\widetilde{R}^{L^{\bot}}(X,Y)\xi:=L^{\bot}_XL^{\bot}_Y\xi-L^{\bot}_YL^{\bot}_X\xi-L^{\bot}_{[X,Y]^D}\xi-\pi^{D^{\bot}}\widetilde{\nabla}_{[X,Y]^{\bot}}\xi.
\end{align}
 Equation (2.26) is called the Ricci
equation for $D$ with respect to $\widetilde{\nabla}$.
\end{thm}
\begin{proof}
From (2.5) and (2.11), we have
\begin{align}
\widetilde{\nabla}_X\widetilde{\nabla}_Y\xi=-\widetilde{\nabla}^D_X(\widetilde{A}_{\xi}Y)-\widetilde{B}(X,\widetilde{A}_{\xi}Y)-\widetilde{A}_{L^{\bot}_Y\xi}X+L^{\bot}_XL^{\bot}_Y\xi,
\end{align}
\begin{align}
\widetilde{\nabla}_Y\widetilde{\nabla}_X\xi=-\widetilde{\nabla}^D_Y(\widetilde{A}_{\xi}X)-\widetilde{B}(Y,\widetilde{A}_{\xi}X)-\widetilde{A}_{L^{\bot}_X\xi}Y+L^{\bot}_YL^{\bot}_X\xi,
\end{align}
\begin{align}
\widetilde{\nabla}_{[X,Y]}\xi=-\widetilde{A}_{\xi}([X,Y]^D)+L^{\bot}_{[X,Y]^D}\xi+\pi^D\widetilde{\nabla}_{[X,Y]^{D^{\bot}}}\xi+\pi^{D^{\bot}}\widetilde{\nabla}_{[X,Y]^{D^{\bot}}}\xi.
\end{align}
From (2.27)-(2.30), we get (2.26).
\end{proof}

\begin{cor} If $U=0$, then we have
\begin{align}
({R}(X,Y)\xi)^{D^{\bot}}=-{B}(X,{A}_{\xi}Y)+{B}(Y,{A}_{\xi}X)+{R}^{L^{\bot}}(X,Y)\xi.
\end{align}
\end{cor}

Nextly, we prove the Chen inequality with respect to $D$ and $\widetilde{\nabla}.$
For $X,Y\in\Gamma(TM)$, we let
\begin{equation*}
\alpha(X,Y)=({{\nabla}}_X\omega)(Y)-\omega(X)\omega(Y)+\frac{1}{2}g(X,Y)\omega(U).
\end{equation*}
From \cite{MO1}, we have
\begin{align}
\widetilde{R}(X,Y,Z,W)&={R^L}(X,Y,Z,W)+\alpha(X,Z)g(Y,W)-\alpha(Y,Z)g(X,W)\\
&~+g(X,Z)\alpha(Y,W)-g(Y,Z)\alpha(X,W).\notag
\end{align}
In $M$ we can choose a local orthonormal frame $E_1,\cdots,E_n,E_{n+1},\cdots,E_{m},$ such that, $E_1,\cdots,E_n$ are orthonormal frames of $D$. We write
$\lambda=\sum_{j=1}^n\alpha(E_j,E_j)$ and
 $h_{ij}^r=g(B(E_i,E_j),E_r)$ and
 $\widetilde{h}_{ij}^r=g(\widetilde{B}(E_i,E_j),E_r)$
 for $1\leq i,j\leq n$ and $n+1\leq r\leq m.$
The squared length of $B$ is $||B||^2=\sum_{i,j=1}^ng(B(E_i,E_j),B(E_i,E_j))$ and the squared length of $\widetilde{B}$ is $||\widetilde{B}||^2=\sum_{i,j=1}^ng(\widetilde{B}(E_i,E_j),\widetilde{B}(E_i,E_j)).$
Let $M$ be an $m$-dimensional real space form of constant sectional curvature $c$ endowed with a semi-symmetric connection $\widetilde{\nabla}$. The curvature tensor
$R^L$ with respect to the Levi-Civita connection on $M$ is expressed by
\begin{equation}
{R^L}(X,Y,Z,W)=c\{g(X,W)g(Y,Z)-g(X,Z)g(Y,W)\}.
\end{equation}
By (2.32) and (2.33), we get
\begin{align}
\widetilde{R}(X,Y,Z,W)&=c\{g(X,W)g(Y,Z)-g(X,Z)g(Y,W)\}+\alpha(X,Z)g(Y,W)\\
&-\alpha(Y,Z)g(X,W)+g(X,Z)\alpha(Y,W)-g(Y,Z)\alpha(X,W).\notag
\end{align}
\indent Let $\Pi\subset D$, be a $2$-plane section. Denote by $\widetilde{K}^D(\Pi)$ the sectional curvature of $D$ with the induced connection $\widetilde{\nabla}^D$ defined by
\begin{align}\widetilde{K}^D(\Pi)=\frac{1}{2}[\widetilde{R}^D(E_1,E_2,E_2,E_1)-\widetilde{R}^D(E_1,E_2,E_1,E_2)],
\end{align}
where $E_1,E_2$ are orthonormal basis of $\Pi$ and $\widetilde{K}^D(\Pi)$ is independent of the choice of $e_1,e_2$.
For any orthonormal basis $\{E_1,\cdots,E_n\}$ of
$D$, the scalar curvature $\widetilde{\tau}^D$ with respect to $D$ and $\widetilde{\nabla}^D$ is defined by
\begin{align}
\widetilde{\tau}^D=\frac{1}{2}\sum_{1\leq i,j\leq n}\widetilde{R}^D(E_i,E_j,E_j,E_i).
\end{align}
By Theorem 2.4, we have
\begin{align}
\widetilde{R}(X,Y,Z,W)&=\widetilde{R}^D(X,Y,Z,W)+g(\widetilde{B}(Y,W),\widetilde{B}(X,Z))\\
&-g(\widetilde{B}(X,W),\widetilde{B}(Y,Z))
+g(B(Z,W),[X,Y]).\notag
\end{align}
Let $E_1,E_2$ be the orthonormal basis of $\Pi\subset D$ and define
\begin{align}
A^D=&\frac{1}{2}\sum_{1\leq i,j\leq n}g(B(E_j,E_i),[E_j,E_i]),\\
\Omega^{\Pi}=&\alpha(E_1,E_1)+\alpha(E_2,E_2)-\frac{1}{2}g(B(E_1,E_2)-B(E_2,E_1),[E_1,E_2]).\notag
\end{align}
Then $A^D$ and $\Omega^{\Pi}$ are independent of the choice of the orthonormal basis.
For the distribution $D$ of the real space form $M$ endowed with a semi-symmetric metric connection, we establish the following inequality, which we called the Chen first inequality.\\

\begin{thm}
 Let $TM=D\oplus D^{\bot}$, ${\rm dim}D=n\geq 3$, and let $M$ be a manifold with constant sectional curvature $c$ endowed with a connection
$\widetilde{\nabla}$, then
\begin{align}
\widetilde{\tau}^D-\widetilde{K}^D(\Pi)&\leq \frac{(n+1)(n-2)}{2}c-(n-1)\lambda+A^D+\Omega^{\Pi}+\frac{n^2(n-2)}{2(n-1)}\|\widetilde{H}\|^2+\frac{1}{2}||\widetilde{B}||^2.
\end{align}
\end{thm}

\begin{proof}
Let $\{E_1,\cdots,E_n\}$ and $\{E_{n+1},\cdots,E_{m}\}$ be orthonormal basis of $D$ and $D^\bot$ respectively. Let $E_1,E_2$ be the orthonormal basis of $\Pi\subset D$.
By (2.34), we obtain
\begin{align}
\widetilde{R}(E_1,E_2,E_1,E_2)=-c+\alpha(E_1,E_1)+\alpha(E_2,E_2).
\end{align}
By (2.37), we have
\begin{align}
\widetilde{R}(E_1,E_2,E_1,E_2)&=\widetilde{R}^D(E_1,E_2,E_1,E_2)+g(\widetilde{B}(E_2,E_2),\widetilde{B}(E_1,E_1))\\
&-g(\widetilde{B}(E_1,E_2),\widetilde{B}(E_2,E_1))
+g(B(E_1,E_2),[E_1,E_2]).\notag
\end{align}
By (2.40) and (2.41), we obtain
\begin{align}
\widetilde{R}^D(E_1,E_2,E_1,E_2)=&-c+\alpha(E_1,E_1)+\alpha(E_2,E_2)\\
&-g(B(E_1,E_2),[E_1,E_2])
-\sum_{r=n+1}^{m}[\widetilde{h}^r_{11}\widetilde{h}^r_{22}-\widetilde{h}^r_{12}\widetilde{h}^r_{21}]
.\notag
\end{align}
Similarly to (2.42), we have
\begin{align}
\widetilde{R}^D(E_1,E_2,E_2,E_1)=&c-\alpha(E_1,E_1)-\alpha(E_2,E_2)\\
&-g(B(E_2,E_1),[E_1,E_2])
+\sum_{r=n+1}^{m}[\widetilde{h}^r_{11}\widetilde{h}^r_{22}-\widetilde{h}^r_{12}\widetilde{h}^r_{21}]
.\notag
\end{align}
By (2.35),(2.42) and (2.43), we obtain
\begin{align}
\widetilde{K}^D(\Pi)=&c-\Omega^{\Pi}
+\sum_{r=n+1}^{m}[\widetilde{h}^r_{11}\widetilde{h}^r_{22}-\widetilde{h}^r_{12}\widetilde{h}^r_{21}].
\end{align}
Similarly to (2.43), we have
\begin{align}
\widetilde{R}^D(E_i,E_j,E_j,E_i)=&c-\alpha(E_j,E_j)-\alpha(E_i,E_i)
-g(\widetilde{B}(E_j,E_i),\widetilde{B}(E_i,E_j))\\
&+g(\widetilde{B}(E_i,E_i),\widetilde{B}(E_j,E_j))
-g(B(E_j,E_i),[E_i,E_j]).\notag
\end{align}
Then
\begin{align}
\widetilde{\tau}^D&=\frac{1}{2}\sum_{1\leq i\neq j\leq n}{\widetilde{R}^D}(E_i,E_j,E_j,E_i)\\
&=\frac{n(n-1)}{2}c-(n-1)\lambda+A^D+\sum_{r=n+1}^{m}\sum_{1\leq i<j\leq n}[\widetilde{h}^r_{ii}\widetilde{h}^r_{jj}-\widetilde{h}^r_{ij}\widetilde{h}^r_{ji}].\notag
\end{align}
So
\begin{align}
\widetilde{\tau}^D-\widetilde{K}^D(\Pi)&=\frac{(n+1)(n-2)}{2}c-(n-1)\lambda+A^D+\Omega^{\Pi}\\
&~+\sum_{r=n+1}^{m}[\sum_{1\leq i<j\leq n}\widetilde{h}^r_{ii}\widetilde{h}^r_{jj}-\widetilde{h}^r_{11}\widetilde{h}^r_{22}-\sum_{1\leq i<j\leq n}\widetilde{h}^r_{ij}\widetilde{h}^r_{ji}+\widetilde{h}^r_{12}
\widetilde{h}^r_{21}]\notag\\
&=\frac{(n+1)(n-2)}{2}c-(n-1)\lambda+A^D+\Omega^{\Pi}\notag\\
&~+\sum_{r=n+1}^{m}[(\widetilde{h}^r_{11}+\widetilde{h}^r_{22})\sum_{3\leq j\leq n}\widetilde{h}^r_{jj}
+\sum_{3\leq i<j\leq n}\widetilde{h}^r_{ii}\widetilde{h}^r_{jj}
-\sum_{1\leq i<j\leq n}\widetilde{h}^r_{ij}\widetilde{h}^r_{ji}+\widetilde{h}^r_{12}
\widetilde{h}^r_{21}]\notag
\end{align}
By Lemma 2.4 in \cite{ZZS}, we get
\begin{equation} \label{C4}
\sum_{r=n+1}^{m}[(\widetilde{h}^r_{11}+\widetilde{h}^r_{22})\sum_{3\leq j\leq n}\widetilde{h}^r_{jj}
+\sum_{3\leq i<j\leq n}\widetilde{h}^r_{ii}\widetilde{h}^r_{jj}]
\leq \frac{n^2(n-2)}{2(n-1)}\|\widetilde{H}\|^2.
\end{equation}
We note that
\begin{align}
&~\sum_{r=n+1}^{m}[
-\sum_{1\leq i<j\leq n}\widetilde{h}^r_{ij}\widetilde{h}^r_{ji}+\widetilde{h}^r_{12}
\widetilde{h}^r_{21}]\\
&=\sum_{r=n+1}^{m}[-\sum_{3\leq j\leq n}\widetilde{h}^r_{1j}\widetilde{h}^r_{j1}
-\sum_{2\leq i<j\leq n}\widetilde{h}^r_{ij}\widetilde{h}^r_{ji}]\notag\\
&\leq\sum_{r=n+1}^{m}[\sum_{3\leq j\leq n}\frac{(\widetilde{h}^r_{1j})^2+(\widetilde{h}^r_{j1})^2}{2}
+\sum_{2\leq i<j\leq n}\frac{(\widetilde{h}^r_{ij})^2+(\widetilde{h}^r_{ji})^2}{2}]\notag\\
&\leq\sum_{r=n+1}^{m}[\sum_{3\leq j\leq n}\frac{(\widetilde{h}^r_{1j})^2+(\widetilde{h}^r_{j1})^2}{2}
+\sum_{2\leq i<j\leq n}\frac{(\widetilde{h}^r_{ij})^2+(\widetilde{h}^r_{ji})^2}{2}
+\sum_{i=1}^n\frac{(\widetilde{h}^r_{ii})^2}{2}+\frac{(\widetilde{h}^r_{12})^2+(\widetilde{h}^r_{21})^2}{2}
]\notag\\
&=\frac{\|\widetilde{B}\|^2}{2}.\notag
\end{align}
By (2.47)-(2.49), we get (2.39).
\end{proof}

\begin{cor} If $U\in\Gamma(D)$, then $H= \widetilde{H}$. In this case, the inequality in Theorem 2.10 becomes
\begin{align}
\widetilde{\tau}^D-\widetilde{K}^D(\Pi)&\leq \frac{(n+1)(n-2)}{2}c-(n-1)\lambda+A^D+\Omega^{\Pi}+\frac{n^2(n-2)}{2(n-1)}\|{H}\|^2+\frac{1}{2}||{B}||^2.
\end{align}
\end{cor}
\begin{cor}
 The equality case of (2.39) holds if and only if $D$ is totally geodesic with respect to $\widetilde{\nabla}$ and $\widetilde{h}^r_{12}=\widetilde{h}^r_{21}=0.$
\end{cor}
\begin{proof}
The equality case of (2.49) holds if and only if $\widetilde{h}^r_{ii}=0$, for $1\leq i\leq n$,
$\widetilde{h}^r_{12}=\widetilde{h}^r_{21}=0$ and $\widetilde{h}^r_{1j}=-\widetilde{h}^r_{j1}$ for $3\leq j \leq n$ and $\widetilde{h}^r_{kl}=-\widetilde{h}^r_{lk}$
for $2\leq k<l \leq n$.\\
\indent The equality case of (2.48) holds if and only if $\widetilde{h}^r_{11}+\widetilde{h}^r_{22}=\widetilde{h}^r_{ii}$ for $3\leq i \leq n$. So Corollary 2.12 holds.
\end{proof}

For each unit vector field $X\in\Gamma(D)$, we choose the orthonormal basis $\{E_1,\cdots, E_n\}$ of $D$ such that $E_1=X$. We define
\begin{align}
&\widetilde{{\rm Ric}}^D(X)=\sum_{j=2}^n\widetilde{R}^D(X,E_j,E_j,X);~~~A^D(X)=\sum_{j=2}^ng(B(E_j,X),[E_j,X])\\
&\|\widetilde{B}^X\|^2=\sum_{j=2}^n[g(\widetilde{B}(X,E_j),\widetilde{B}(X,E_j))+g(\widetilde{B}(E_j,X),\widetilde{B}(E_j,X))].\notag
\end{align}

\begin{thm}
 Let $TM=D\oplus D^{\bot}$, ${\rm dim}D=n\geq 2$, and let $M$ be a manifold with constant sectional curvature $c$ endowed with a connection
$\widetilde{\nabla}$, then
\begin{align}
\widetilde{{\rm Ric}}^D(X)&\leq (n-1)c-\lambda+(2-n)\alpha(X,X)
+\frac{n^2}{4}\|\widetilde{H}\|^2+\frac{\|\widetilde{B}^X\|^2}{2}+A^D(X).
\end{align}
\end{thm}
\begin{proof}
Similarly to (2.45), we have
\begin{align}
\widetilde{{\rm Ric}}^D(X)&\leq (n-1)c-\lambda+(2-n)\alpha(X,X)
+\sum_{r=n+1}^{m}\sum_{j=2}^n[\widetilde{h}^r_{11}\widetilde{h}^r_{jj}-\widetilde{h}^r_{1j}\widetilde{h}^r_{j1}]
+A^D(X).
\end{align}
By Lemma 2.5 in \cite{ZZS}, we get
\begin{equation}
\sum_{r=n+1}^{n+p}\sum_{j=2}^n\widetilde{h}^r_{11}\widetilde{h}^r_{jj}\leq \frac{n^2}{4}\|\widetilde{H}\|^2.
\end{equation}
We note that
\begin{align}
-\sum_{r=n+1}^{m}\sum_{j=2}^n\widetilde{h}^r_{1j}\widetilde{h}^r_{j1}\leq \sum_{r=n+1}^{m}\sum_{j=2}^n \frac{(\widetilde{h}^r_{1j})^2+(\widetilde{h}^r_{j1})^2}{2}=\frac{\|\widetilde{B}^X\|^2}{2}.
\end{align}
By (2.53)-(2.55), we get (2.52).
\end{proof}

\begin{cor}
 The equality case of (2.52) holds if and only if $\widetilde{h}^r_{1j}=-\widetilde{h}^r_{j1}$ for $2\leq j\leq n$ and $\widetilde{h}^r_{11}-\widetilde{h}^r_{22}-\cdots-\widetilde{h}^r_{nn}=0.$
\end{cor}

\vskip 1 true cm

\section{ Non integrable distributions with a semi-symmetric non-metric connection}

\indent Let $U\in \Gamma(TM)$ be a vector field and $\omega$ be a $1$-form defined by $\omega(V)=g(U,V)$ for any $V\in \Gamma(TM)$. We define the semi-symmetric non-metric connection on $M$
\begin{equation}
\widehat{\nabla}_XY=\nabla_{X}Y+\omega(Y)X.
\end{equation}
 Let for $X,Y\in\Gamma(D)$
\begin{equation}
\widehat{\nabla}_XY=\widehat{\nabla}^D_{X}Y+\widehat{B}(X,Y),~~\widehat{\nabla}^D_{X}Y=\pi^D\widehat{\nabla}_{X}Y,~~\widehat{B}(X,Y)=\pi^{D^\bot}\widehat{\nabla}_XY.
\end{equation}
We call the $\widehat{B}(X,Y)$ the second fundamental form with respect to the semi-symmetric non-metric connection. By (3.1) and (3.2), we have
\begin{equation}
\widehat{\nabla}^D_{X}Y={\nabla}^D_{X}Y+\omega(Y)X,~~\widehat{B}(X,Y)={B}(X,Y).
\end{equation}
By (3.3), we have

\begin{align}
&\widehat{\nabla}^D_{X}(g^D)(Y,Z)=-\omega(Y)g^D(X,Z)-\omega(Z)g^D(X,Y),\\
&\widehat{T}^D(X,Y)=-[X,Y]^{D^\bot}+\omega(Y)X-\omega(X)Y.\notag
\end{align}
Similarly to the case $D=TM$, we have
\begin{thm}
There exists a unique linear connection $\widehat{\nabla}^D: \Gamma(D)\times \Gamma(D)\rightarrow \Gamma(D)$ on $D$, which satisfies the property (3.4).
\end{thm}
We may define the mean curvature vector and minimal distributions and totally geodesic distributions with respect to $\widehat{\nabla}.$ We have

\begin{prop}
$D$ is minimal (resp. totally geodesic, umbilical) with respect to
$\nabla$ if and only if  $D$ is minimal (resp. totally geodesic, umbilical) with respect to
$\widehat{\nabla}$.
\end{prop}

Let
\begin{equation}
\widehat{\nabla}_X\xi=-\widehat{A}_\xi X+L^{\bot}_X\xi,
\end{equation}
where $\widehat{A}_{\xi}=(A_{\xi}-\omega(\xi))I$.
Similarly to (2.12) and (2.13), we may define $\widehat{R}$ and $\widehat{R}^D,$ then similarly to Theorems 2.4, 2.6, 2.8, we have

\begin{thm}
Given $X,Y,Z,W\in\Gamma(D)$ and $\xi\in\Gamma(D)$, we have
\begin{align}
\widehat{R}(X,Y,Z,W)&=\widehat{R}^D(X,Y,Z,W)-g(B(X,W),B(Y,Z))+g(B(Y,W),B(X,Z))\\
&+g(X,W)\omega(B(Y,Z))-g(Y,W)\omega(B(X,Z))+g(B(Z,W),[X,Y]).\notag
\end{align}

\begin{align}
(\widehat{R}(X,Y)Z)^{D^{\bot}}=&(\widehat{L}^{\bot}_X{B})(Y,Z)-(\widehat{L}^{\bot}_Y{B})(X,Z)\\
&-\omega(X){B}(Y,Z)+\omega(Y){B}(X,Z)-\pi^{D^{\bot}}[{[X,Y]^{D^{\bot}}},Z]\notag\\
&-L^{\bot}_Z([X,Y]^{D^{\bot}})-\omega(Z){[X,Y]^{D^{\bot}}},\notag
\end{align}
where $(\widehat{L}^{\bot}_X{B})(Y,Z)=L^{\bot}_X({B}(Y,Z))-{B}(\widehat{\nabla}^D_XY,Z)-{B}(Y,\widehat{\nabla}^D_XZ).$
\begin{align}
(\widehat{R}(X,Y)\xi)^{D^{\bot}}=-{B}(X,\widehat{A}_{\xi}Y)+{B}(Y,\widehat{A}_{\xi}X)+\widehat{R}^{L^{\bot}}(X,Y)\xi,
\end{align}
where
\begin{align}
\widehat{R}^{L^{\bot}}(X,Y)\xi:=L^{\bot}_XL^{\bot}_Y\xi-L^{\bot}_YL^{\bot}_X\xi-L^{\bot}_{[X,Y]^D}\xi-\pi^{D^{\bot}}\widehat{\nabla}_{[X,Y]^{\bot}}\xi.\notag
\end{align}
\end{thm}

We may define $\widehat{K}^D(\Pi)$, $\widehat{\tau}^D$ similarly.
For $X,Y\in\Gamma(TM)$, we let
\begin{equation*}
\alpha_1(X,Y)=({{\nabla}}_X\omega)(Y)-\omega(X)\omega(Y).
\end{equation*}
Let $\lambda_1=\sum_{j=1}^n\alpha_1(E_j,E_j).$
From \cite{OM}, we have
\begin{align}
\widehat{R}(X,Y,Z,W)&={R^L}(X,Y,Z,W)+\alpha_1(X,Z)g(Y,W)-\alpha_1(Y,Z)g(X,W).
\end{align}
Let $M$ be an $m$-dimensional real space form of constant sectional curvature $c$ endowed with a semi-symmetric connection $\widehat{\nabla}$.
By (2.33) and (3.9), we get

\begin{align}
\widehat{R}(X,Y,Z,W)&=c\{g(X,W)g(Y,Z)-g(X,Z)g(Y,W)\}\\
&+\alpha_1(X,Z)g(Y,W)-\alpha_1(Y,Z)g(X,W).\notag
\end{align}
Let
\begin{align}
&{\rm tr}(\alpha_1|_{\Pi})=\alpha_1(E_1,E_1)+\alpha_1(E_2,E_2),~~{\rm tr}(B|_{\Pi})=B(E_1,E_1)+B(E_2,E_2),\\
&\Omega^{\Pi*}=-\frac{1}{2}g(B(E_1,E_2)-B(E_2,E_1),[E_1,E_2]).\notag
\end{align}

\begin{thm}
 Let $TM=D\oplus D^{\bot}$, ${\rm dim}D=n\geq 3$, and let $M$ be a manifold with constant sectional curvature $c$ endowed with a connection
$\widehat{\nabla}$, then
\begin{align}
\widehat{\tau}^D-\widehat{K}^D(\Pi)&\leq \frac{(n+1)(n-2)}{2}c-\frac{n-1}{2}\lambda_1-\frac{n(n-1)}{2}\omega(H)\\
&+\frac{1}{2}{\rm tr}(\alpha_1|_{\Pi})
+\frac{1}{2}\omega({\rm tr}(B|_{\Pi}))
+A^D+\Omega^{\Pi*}
+\frac{n^2(n-2)}{2(n-1)}\|{H}\|^2+\frac{1}{2}||{B}||^2.\notag
\end{align}
\end{thm}

\begin{proof}
Let $\{E_1,\cdots,E_n\}$ and $\{E_{n+1},\cdots,E_{m}\}$ be orthonormal basis of $D$ and $D^\bot$ respectively. Let $E_1,E_2$ be the orthonormal basis of $\Pi\subset D$.
By (3.10), we obtain
\begin{align}
\widehat{R}(E_1,E_2,E_1,E_2)=-c+\alpha_1(E_1,E_1).
\end{align}
By (3.6), we have
\begin{align}
\widehat{R}^D(E_1,E_2,E_1,E_2)=&-c+\alpha_1(E_1,E_1)
+g(B(E_1,E_2),B(E_2,E_1))\\
&-g(B(E_1,E_1),B(E_2,E_2))+\omega(B(E_1,E_1))
-g(B(E_1,E_2),[E_1,E_2]),\notag
\end{align}
Similarly, we have
\begin{align}
\widehat{R}^D(E_1,E_2,E_2,E_1)=&c-\alpha_1(E_2,E_2)
-g(B(E_1,E_2),B(E_2,E_1))\\
&+g(B(E_1,E_1),B(E_2,E_2))-\omega(B(E_2,E_2))
-g(B(E_2,E_1),[E_1,E_2]).\notag
\end{align}
So we obtain
\begin{align}
\widehat{K}^D(\Pi)=&c-\frac{1}{2}{\rm tr}(\alpha_1|_{\Pi})-\frac{1}{2}g({\rm tr}(B|_{\Pi}),U)
-\Omega^{\Pi^*}
+\sum_{r=n+1}^{m}[{h}^r_{11}{h}^r_{22}-{h}^r_{12}{h}^r_{21}].
\end{align}
Similarly to (3.15), we have
\begin{align}
\widehat{R}^D(E_i,E_j,E_j,E_i)=&c-\alpha_1(E_j,E_j)
-g(B(E_i,E_j),B(E_j,E_i))\\
&+g(B(E_i,E_i),B(E_j,E_j))-\omega(B(E_j,E_j))
-g(B(E_j,E_i),[E_i,E_j]).\notag
\end{align}
Then
\begin{align}
\widehat{\tau}^D&=\frac{1}{2}\sum_{1\leq i\neq j\leq n}{\widehat{R}^D}(E_i,E_j,E_j,E_i)\\
&=\frac{n(n-1)}{2}c-\frac{n-1}{2}\lambda-\frac{n(n-1)}{2}\omega(H)+
A^D+\sum_{r=n+1}^{m}\sum_{1\leq i<j\leq n}[{h}^r_{ii}{h}^r_{jj}-{h}^r_{ij}{h}^r_{ji}].\notag
\end{align}
So
\begin{align}
\widehat{\tau}^D-\widehat{K}^D(\Pi)&=\frac{(n+1)(n-2)}{2}c-\frac{n-1}{2}\lambda-\frac{n(n-1)}{2}\omega(H)\\
&
+\frac{1}{2}{\rm tr}(\alpha_1|_{\Pi})+\frac{1}{2}g({\rm tr}(B|_{\Pi}),U)
+A^D+\Omega^{\Pi^*}\notag\\
&~+\sum_{r=n+1}^{m}[\sum_{1\leq i<j\leq n}{h}^r_{ii}{h}^r_{jj}-{h}^r_{11}{h}^r_{22}-\sum_{1\leq i<j\leq n}{h}^r_{ij}{h}^r_{ji}+{h}^r_{12}
{h}^r_{21}].\notag\\
\end{align}
By (2.48) and (2.49), we get (3.12).
\end{proof}

\begin{cor}
 The equality case of (3.12) holds if and only if $D$ is totally geodesic with respect to ${\nabla}$ and ${h}^r_{12}={h}^r_{21}=0.$
\end{cor}

\begin{thm}
 Let $TM=D\oplus D^{\bot}$, ${\rm dim}D=n\geq 2$, and let $M$ be a manifold with constant sectional curvature $c$ endowed with a connection
$\widehat{\nabla}$, then
\begin{align}
\widehat{{\rm Ric}}^D(X)&\leq (n-1)c-\lambda+\alpha_1(X,X)-n\omega(H)\\
&+\omega(B(X,X))
+\frac{n^2}{4}\|{H}\|^2+\frac{\|{B}^X\|^2}{2}+A^D(X).\notag
\end{align}
\end{thm}

\begin{proof}
By (3.17), we have

\begin{align}
\widehat{{\rm Ric}}^D(X)&\leq (n-1)c-\lambda+\alpha_1(X,X)-n\omega(H)\\
&+\omega(B(X,X))
+A^D(X)
+\sum_{r=n+1}^{m}\sum_{j=2}^n[{h}^r_{11}{h}^r_{jj}-{h}^r_{1j}{h}^r_{j1}].\notag
\end{align}
By (2.54) and (2.55), we get (3.21).
\end{proof}

\begin{cor}
 The equality case of (3.21) holds if and only if ${h}^r_{1j}=-{h}^r_{j1}$ for $2\leq j\leq n$ and ${h}^r_{11}-{h}^r_{22}-\cdots-{h}^r_{nn}=0.$
\end{cor}

\vskip 1 true cm

\section{ Non integrable distributions with a statistical connection}

Let $(M,g,\overline{\nabla})$ denote a statistical manifold and $(M,g,\overline{\nabla}^*)$ denote a dual statistical manifold. Then
\begin{align}
\overline{\nabla}_XY=\nabla_XY+K(X,Y);~~~\overline{\nabla}^*_XY=\nabla_XY-K(X,Y),
\end{align}
where $K:\Gamma(TM)\times \Gamma(TM)\rightarrow \Gamma(TM)$ is a tensor field. Let $C(X,Y,Z)=g(K(X,Y),Z)$, then $C(X,Y,Z)$ is a symmetric tensor. Let
\begin{align}
\overline{\nabla}_XY=\overline{\nabla}^D_XY+\overline{B}(X,Y);~~~\overline{\nabla}^*_XY=\overline\nabla^{D,*}_XY+\overline{B}^*(X,Y),
\end{align}
where
\begin{align}
&\overline{\nabla}^D_XY={\nabla}^D_XY+\pi^DK(X,Y);~~~
\overline{\nabla}^{D,*}_XY={\nabla}^D_XY-\pi^DK(X,Y).\\
&\overline{B}(X,Y)=B(X,Y)+\pi^{D^{\bot}}K(X,Y);~~~
\overline{B}^*(X,Y)=B(X,Y)-\pi^{D^{\bot}}K(X,Y)\notag
\end{align}
Let $\overline{A}_{\xi},~~\overline{A}^*_{\xi}:\Gamma(D)\rightarrow\Gamma(D)$ defined by
\begin{align}
g(\overline{A}_{\xi}X,Y)=g(\overline{B}(X,Y),\xi);~~~g(\overline{A}^*_{\xi}X,Y)=g(\overline{B}^*(X,Y),\xi),
\end{align}
where $\overline{A}_{\xi}$, $\overline{A}^*_{\xi}$ are called shape operators on $D$ with respect to $\overline{\nabla},$ $\overline{\nabla}^*$ respectively.
Then $\overline{A}_{\xi}X=-\pi^D\overline{\nabla}^*_X\xi$, $\overline{A}^*_{\xi}X=-\pi^D\overline{\nabla}_X\xi$
Let
\begin{align}
\overline{\nabla}_X\xi=-\overline{A}^*_{\xi}X+\overline{L}^{\bot}_X\xi;~~\overline{\nabla}^*_X\xi=-\overline{A}_{\xi}X+\overline{L}^{\bot,*}_X\xi,
\end{align}
which we called the Weingarten formulas. $\overline{L}^{\bot}_X\xi,~~\overline{L}^{\bot,*}_X\xi:\Gamma(D)\times \Gamma(D^{\bot})\rightarrow \Gamma(D^{\bot})$ are dual connections on $D^{\bot}$ along $\Gamma(D)$.
Similarly to (2.12) and (2.13), we may define $\overline{R},\overline{R}^D,\overline{R}^*,\overline{R}^{D,*}$.

\begin{thm} Given $X,Y,Z,W\in\Gamma(D)$, we have
\begin{align}
\overline{R}(X,Y,Z,W)&=\overline{R}^D(X,Y,Z,W)+g(\overline{B}^*(Y,W),\overline{B}(X,Z))\\
&-g(\overline{B}^*(X,W),\overline{B}(Y,Z))
+g(\overline{B}^*(Z,W),[X,Y]).\notag
\end{align}
\end{thm}

\begin{proof}
From Equations (4.2) and (4.5), we have for $X,Y,Z\in\Gamma(D)$
\begin{align}
\overline{\nabla}_X\overline{\nabla}_YZ&=\overline{\nabla}^D_X\overline{\nabla}^D_YZ+\overline{B}(X,\overline{\nabla}^D_YZ)\\
&-\overline{A}^*_{\overline{B}(Y,Z)}X+\overline{L}^{\bot}_X(\overline{B}(Y,Z))\notag,
\end{align}
\begin{align}
\overline{\nabla}_Y\overline{\nabla}_XZ&=\overline{\nabla}^D_Y\overline{\nabla}^D_XZ+\overline{B}(Y,\overline{\nabla}^D_XZ)\\
&-\overline{A}^*_{\overline{B}(X,Z)}Y+\overline{L}^{\bot}_Y(\overline{B}(X,Z))\notag,
\end{align}
\begin{align}
\overline{\nabla}_{[X,Y]}Z&=\overline{\nabla}^D_{[X,Y]^{D}}Z+\overline{B}([X,Y]^{D},Z)
-\overline{A}^*_{[X,Y]^{D^{\bot}}}Z\\
&+\overline{L}^{\bot}_Z([X,Y]^{D^{\bot}})+[{[X,Y]^{D^{\bot}}},Z].\notag
\end{align}
So we get
\begin{align}
\overline{R}(X,Y)Z=&\overline{R}^D(X,Y)Z-\pi^{D^{\bot}}[{[X,Y]^{D^{\bot}}},Z]+\overline{B}(X,\overline{\nabla}^D_YZ)\\
&-\overline{B}(Y,\overline{\nabla}^D_XZ)
-\overline{B}([X,Y]^{D},Z)-\overline{A}^*_{\overline{B}(Y,Z)}X\notag\\
&+\overline{A}^*_{\overline{B}(X,Z)}Y
+\overline{L}^{\bot}_X(\overline{B}(Y,Z))-\overline{L}^{\bot}_Y(\overline{B}(X,Z))\notag\\
&+\overline{A}^*_{[X,Y]^{D^{\bot}}}Z
-\overline{L}^{\bot}_Z([X,Y]^{D^{\bot}}).\notag
\end{align}
By (4.4) and (4.10), we get (4.6).
\end{proof}
By (4.10), we have
\begin{thm} Given $X,Y,Z\in\Gamma(D)$, the following equality holds:
\begin{align}
[\overline{R}(X,Y)Z]^{D^{\bot}}=&-\pi^{D^{\bot}}[{[X,Y]^{D^{\bot}}},Z]+\overline{B}(X,\overline{\nabla}^D_YZ)\\
&-\overline{B}(Y,\overline{\nabla}^D_XZ)
-\overline{B}([X,Y]^{D},Z)\notag\\
&+\overline{L}^{\bot}_X(\overline{B}(Y,Z))-\overline{L}^{\bot}_Y(\overline{B}(X,Z))
-\overline{L}^{\bot}_Z([X,Y]^{D^{\bot}}).\notag
\end{align}
\end{thm}

From (4.2), (4.4) and (4.5), we get
\begin{thm}
Given $X,Y\in\Gamma(D)$, $\xi,\eta\in\Gamma(D^{\bot})$, we have
\begin{align}
g(\overline{R}(X,Y)\xi,\eta)=g(\overline{A}_{\eta}Y, \overline{A}^*_{\xi}X))
-g(\overline{A}_{\eta}X, \overline{A}^*_{\xi}Y))
+g(\overline{R}^{L^{\bot}}(X,Y)\xi,\eta),
\end{align}
where
\begin{align}
\overline{R}^{L^{\bot}}(X,Y)\xi:=\overline{L}^{\bot}_X\overline{L}^{\bot}_Y\xi-\overline{L}^{\bot}_Y\overline{L}^{\bot}_X\xi
-\overline{L}^{\bot}_{[X,Y]^D}\xi-\pi^{D^{\bot}}\overline{\nabla}_{[X,Y]^{\bot}}\xi.
\end{align}
\end{thm}

\vskip 1 true cm

\section{Examples}
\noindent{\bf Example 1}\\
 \indent Consider the $3$-dimensional unit sphere $\mathbb{S}^3$ as a Riemannian manifold endowed with the metric induced from $\mathbb{R}^4$. The tangent space of  $\mathbb{S}^3$
at each point has an orthonormal basis given by the vector fields
\begin{align}
X_1=x_2\partial_{x_1}-x_1\partial_{x_2}-x_4\partial_{x_3}+x_3\partial_{x_4}\notag\\
X_2=x_4\partial_{x_1}-x_3\partial_{x_2}+x_2\partial_{x_3}-x_1\partial_{x_4}\notag\\
X_3=x_3\partial_{x_1}+x_4\partial_{x_2}-x_1\partial_{x_3}-x_2\partial_{x_4}.\notag
\end{align}
Then
\begin{align}
[X_1,X_2]=2X_3,~~~[X_1,X_3]=-2X_2,~~~[X_2,X_3]=2X_1.
\end{align}
Let $\nabla$ be the Levi-Civita connection on $\mathbb{S}^3$. By (5.1) and the Koszul formula, we have
\begin{align}
&\nabla_{X_1}X_2=X_3,~~~\nabla_{X_2}X_1=-X_3,,~~~\nabla_{X_1}X_1=\nabla_{X_2}X_2=\nabla_{X_3}X_3=0
,\\
&\nabla_{X_1}X_3=-X_2,~~~ \nabla_{X_3}X_1=X_2,~~~ \nabla_{X_2}X_3=X_1,~~~ \nabla_{X_3}X_2=-X_1.\notag
\end{align}
Consider the distribution $D_1=span\{X_1,X_2\}$, which is not integrable by (5.1). The metric of $D_1$ is induced by the metric on $\mathbb{S}^3$. Let $U=X_1+X_3$.
By (5.2), we have
\begin{align}
\nabla^{D_1}_{X_i}X_j=0,~~~\forall i,j=1,2,~~B(X_1,X_1)=B(X_2,X_2)=0,\\
~~B(X_1,X_2)=X_3,~~ B(X_2,X_1)=-X_3.\notag
\end{align}
So $D_1$ is totally geodesic distribution with respect to $\nabla$.
By (2.6), we obtain
\begin{equation}
\widetilde{\nabla}^{D_1}_{X}Y={\nabla}^{D_1}_{X}Y+g(X_1,Y)X-g(X,Y)X_1,~~\widetilde{B}(X,Y)={B}(X,Y)-g(X,Y)X_3.
\end{equation}
So
\begin{align}
&\widetilde{\nabla}^{D_1}_{X_1}X_1=0,~~~\widetilde{\nabla}^{D_1}_{X_1}X_2=0,~~~ \widetilde{\nabla}^{D_1}_{X_2}X_1=X_1,~~~\widetilde{\nabla}^{D_1}_{X_2}X_2=-X_1\\
&\widetilde{B}(X_1,X_1)=-X_3,~~~ \widetilde{B}(X_1,X_2)=X_3,~~~\notag\\
& \widetilde{B}(X_2,X_1)=-X_3,~~~ \widetilde{B}(X_2,X_2)=-X_3,~~~\widetilde{H}=-X_3.\notag
\end{align}
By (2.8) and (2.11), we have
\begin{align}
&A_{X_3}X_1=X_2,~~~ A_{X_3}X_2=-X_1,~~~ \widetilde{A}_{X_3}X_1=X_2-X_1,\\
& \widetilde{A}_{X_3}X_2=-X_1-X_2,~~~ L^{\bot}_{X_1}X_3=L^{\bot}_{X_2}X_3=0.\notag
\end{align}
By (2.13),(2.35),(2.36) and (5.5), we have
\begin{align}
\widetilde{R}^{D^1}(X_1,X_2)X_1=-4X_2,~~~ \widetilde{R}^{D^1}(X_1,X_2)X_2=4X_1,~~~  \widetilde{K}^{D^1}(D_1)=4,~~~ \widetilde{\tau}^{D_1}=4.
\end{align}
By (3.2), we have
\begin{equation}
\widehat{\nabla}^{D_1}_{X}Y={\nabla}^{D_1}_{X}Y+g(X_1,Y)X,~~\widehat{B}(X,Y)={B}(X,Y).
\end{equation}
So
\begin{align}
&\widehat{\nabla}^{D_1}_{X_1}X_1=X_1,~~~\widehat{\nabla}^{D_1}_{X_1}X_2=0,~~~\widehat{\nabla}^{D_1}_{X_2}X_1=X_2,~~~\widehat{\nabla}^{D_1}_{X_2}X_2=0\\
&\widehat{B}(X_1,X_1)=0,~~~ \widehat{B}(X_1,X_2)=X_3,~~~\notag\\
&\widehat{B}(X_2,X_1)=-X_3,~~~ \widehat{B}(X_2,X_2)=0,\notag\\
&\widehat{R}^{D^1}(X_1,X_2)X_1=-5X_2,~~~ \widehat{R}^{D^1}(X_1,X_2)X_2=4X_1.\notag
\end{align}
We consider the distribution $D^2=span\{X_1,X_3\}$ and $U=X_2+X_3$, then similarly we get
\begin{align}
&\widetilde{\nabla}^{D_2}_{X_1}X_1=-X_3,~~~\widetilde{\nabla}^{D_2}_{X_1}X_3=X_1,~~~ \widetilde{\nabla}^{D_2}_{X_3}X_1=0,~~~\widetilde{\nabla}^{D_2}_{X_3}X_3=0.\\
&\widetilde{R}^{D^2}(X_1,X_3)X_1=4X_3,~~~ \widetilde{R}^{D^2}(X_1,X_3)X_3=4X_1,~~~  \widetilde{K}^{D^2}(D_2)=0,~~~ \widetilde{\tau}^{D_2}=0.\notag
\end{align}
Similarly, we obtain
\begin{align}
&\widehat{\nabla}^{D_2}_{X_1}X_1=0,~~~\widehat{\nabla}^{D_2}_{X_1}X_3=X_1,~~~ \widehat{\nabla}^{D_2}_{X_3}X_1=0,~~~\widehat{\nabla}^{D_2}_{X_3}X_3=X_3.\\
&\widehat{R}^{D^2}(X_1,X_3)X_1=4X_3,~~~ \widehat{R}^{D^2}(X_1,X_3)X_3=5X_1,~~~  \widehat{K}^{D^2}(D_2)=\frac{1}{2},~~~ \widehat{\tau}^{D_2}=\frac{1}{2}.\notag
\end{align}
\vskip 1 true cm
\noindent {\bf Example 2}\\
\indent Let $M=\mathbb{R}\times \mathbb{S}^3$ and $D^1=span\{X_1,X_2\}$ and $T\mathbb{S}^3=D^1\oplus D^{1,\bot}$. Let $f(t)\in C^{\infty}(\mathbb{R})$ without zero points. Let $\pi_1:
\mathbb{R}\times \mathbb{S}^3\rightarrow \mathbb{R};(t,x)\rightarrow t$ and $\pi_2:
\mathbb{R}\times \mathbb{S}^3\rightarrow \mathbb{S}^3;(t,x)\rightarrow x$. Let
\begin{align}
&g^M_f=\pi^*_1dt^2\oplus f^2\pi^*_2g^{D^1}\oplus \pi^*_2 g^{D^{1,\bot}};\\
&D=\pi^*_1(T\mathbb{R})\oplus \pi^*_2D^1;~~~g^D=\pi^*_1dt^2\oplus f^2\pi^*_2g^{D^1}.\notag
\end{align}
We call $(D,g^D)$ the warped product distribution on $M$. Let $\nabla^f$ be the Levi-Civita connection on $(M,g^M_f)$ and $\partial_t=\frac{\partial}{\partial t}$ and
$\partial_t(f)=f'$, then by the Koszul formula and (5.12), we get
\begin{align}
&\nabla^f_{\partial_t}\partial_t=0,~~~ \nabla^f_{\partial_t}X_1=\frac{f'}{f}X_1,~~~ \nabla^f_{X_1}\partial_t=\frac{f'}{f}X_1,~~~
\nabla^f_{\partial_t}X_2=\frac{f'}{f}X_2,\\
&~~~ \nabla^f_{X_2}\partial_t=\frac{f'}{f}X_2,~~~ \nabla^f_{\partial_t}X_3=\nabla^f_{X_3}\partial_t=0,~~~
\nabla^f_{X_1}X_1=\nabla^f_{X_2}X_2=-ff'\partial_t,\notag\\
&~~~ \nabla^f_{X_1}X_2=X_3,~~~ \nabla^f_{X_2}X_1=-X_3,
~~~ \nabla^f_{X_1}X_3=-\frac{X_2}{f^2},~~~ \nabla^f_{X_3}X_1=(2-\frac{1}{f^2})X_2,\notag\\
&~~~
\nabla^f_{X_2}X_3=\frac{X_1}{f^2},~~~ \nabla^f_{X_3}X_2=(\frac{1}{f^2}-2)X_1,~~~ \nabla^f_{X_3}X_3=0.\notag
\end{align}
So we obtain
\begin{align}
&R^f(\partial_t,X_1)\partial_t=\frac{f''}{f}X_1,~~~ R^f(\partial_t,X_2)\partial_t=\frac{f''}{f}X_2,~~~ R^f(\partial_t,X_3)\partial_t=0,\\
&~~~
R^f(\partial_t,X_1)X_1=-ff''\partial_t,~~~ R^f(\partial_t,X_1)X_2=-\frac{f'}{f}X_3,~~~ R^f(\partial_t,X_1)X_3=f^{-3}f'X_2,\notag\\
&~~~
 R^f(\partial_t,X_2)X_1=\frac{f'}{f}X_3,~~~
R^f(\partial_t,X_2)X_2=-ff''\partial_t, R^f(\partial_t,X_2)X_3=-f^{-3}f'X_1, \notag\\
&R^f(\partial_t,X_3)X_1=2f^{-3}f'X_2,~~~ R^f(\partial_t,X_3)X_2=-2f^{-3}f'X_1,~~~
R^f(\partial_t,X_3)X_3=0,\notag\\
&~~~ R^f(X_1,X_2)\partial_t=\frac{2f'}{f}X_3,~~~ R^f(X_1,X_3)\partial_t=f'f^{-3}X_2,~~~ R^f(X_2,X_3)\partial_t=-f^{-3}f'X_1,~~~\notag\\
&
R^f(X_1,X_2)X_1=[(f')^2+3f^{-2}-4]X_2,~~~ R^f(X_1,X_2)X_2=[-(f')^2-3f^{-2}+4]X_1,\notag\\
&~~~ R^f(X_1,X_2)X_3=-2f^{-1}f'\partial_t,~~~
R^f(X_1,X_3)X_1=-f^{-2}X_3,~~~ R^f(X_1,X_3)X_2=-f^{-1}f'\partial_t,\notag\\
&~~~ R^f(X_1,X_3)X_3=f^{-4}X_1,~~~ R^f(X_2,X_3)X_1=f^{-1}f'\partial_t,~~~\notag\\
&
R^f(X_2,X_3)X_2=-f^{-2}X_3,~~~ R^f(X_2,X_3)X_3=f^{-4}X_2.\notag
\end{align}
By $R^f(X_2,X_3)X_3=f^{-4}X_2$, then $(M,g^M_f)$ is not flat. Let
${\rm Ric}^f(X,Y)=\sum_{k=1}^4g(R^f(X,E_k)Y,E_k)$ where $\{E_k\}$ are the orthonormal basis of $(M,g^M_f)$. Then we have
\begin{align}
&{\rm Ric}^f(\partial_t,\partial_t)=\frac{2f''}{f}, {\rm Ric}^f(X_1,X_1)={\rm Ric}^f(X_2,X_2)=ff''+(f')^2+2f^{-2}-4,\\
 &{\rm Ric}^f(X_3,X_3)=-2f^{-4},
{\rm Ric}^f(\partial_t,X_1)={\rm Ric}^f(\partial_t,X_2)={\rm Ric}^f(\partial_t,X_3)=0,\notag\\
&{\rm Ric}^f(X_1,X_2)={\rm Ric}^f(X_1,X_3)={\rm Ric}^f(X_2,X_3)=0.\notag
\end{align}
Then it is easy to get that $(M,g^M_f)$ is not Einstein. Let $s^f=\sum_k{\rm Ric}^f(E_k,E_k)$, then
\begin{align}
s^f=4f^{-1}f''+2f^{-2}(f')^2-8f^{-2}-2f^{-4}+4.
\end{align}
Let $D=span\{\partial_t,X_1,X_2\}$, by (5.13), we have
\begin{align}
&\nabla^D_{\partial_t}\partial_t=0,~~~ \nabla^D_{\partial_t}X_1=\frac{f'}{f}X_1,~~~ \nabla^D_{X_1}\partial_t=\frac{f'}{f}X_1,~~~
\nabla^D_{\partial_t}X_2=\frac{f'}{f}X_2,\\
&~~~ \nabla^D_{X_2}\partial_t=\frac{f'}{f}X_2,~~~
\nabla^D_{X_1}X_1=\nabla^D_{X_2}X_2=-ff'\partial_t,\notag\\
&~~~ \nabla^D_{X_1}X_2=0,~~~ \nabla^D_{X_2}X_1=0,\notag
\end{align}
and
\begin{align}
&B(\partial_t,\partial_t)=B(X_1,X_1)=B(X_2,X_2)
=B(X_1,\partial_t)=0,\\
&B(\partial_t,X_1)=B(\partial_t,X_2)=B(X_2,\partial_t)=0,\notag\\
&B(X_1,X_2)=X_3,~~~ B(X_2,X_1)=-X_3,\notag\\
&A_{X_3}\partial_t=0,~~~A_{X_3}X_1=\frac{X_2}{f^2},~~~A_{X_3}X_2=-\frac{X_1}{f^2},\notag\\
&L^{\bot}_{\partial_t}X_3=L^{\bot}_{X_1}X_3=L^{\bot}_{X_2}X_3=0.\notag
\notag
\end{align}
Similarly to (2.13), we can compute the curvature tensor on $D$ as follows:
\begin{align}
&R^D(\partial_t,X_1)\partial_t=\frac{f''}{f}X_1,~~~ R^D(\partial_t,X_2)\partial_t=\frac{f''}{f}X_2,\\
&
R^D(\partial_t,X_1)X_1=-ff''\partial_t,~~~ R^D(\partial_t,X_1)X_2=R^D(\partial_t,X_2)X_1=0,~~~ \notag\\
&
R^D(\partial_t,X_2)X_2=-ff''\partial_t,~~~ R^D(X_1,X_2)\partial_t=0,\notag\\
&
R^D(X_1,X_2)X_1=[(f')^2-4]X_2,~~~ R^D(X_1,X_2)X_2=[-(f')^2+4]X_1.\notag
\end{align}
Let $\widetilde{X_1}=f^{-1}X_1$, $\widetilde{X_2}=f^{-1}X_2.$ We use the similar definition of (2.35) and get
\begin{align}
K^D(\partial_t\wedge \widetilde{X_1})=-\frac{f''}{f},~~~K^D(\partial_t\wedge \widetilde{X_2})=-\frac{f''}{f},~~~K^D(\widetilde{X_1}\wedge \widetilde{X_2})=\frac{4-(f')^2}{f^2}.
\end{align}
We define the Ricci tensor of $D$ by ${\rm Ric}^D(X,Y)=\sum_{k=1}^3g^D(R^D(X,E_k)Y,E_k)$ where $X,Y\in\Gamma(D)$ and $E_1,E_2,E_3$ are orthonormal basis of $(D,g^D)$. Then
\begin{align}
&{\rm Ric}^D(\partial_t,\partial_t)=\frac{2f''}{f}, {\rm Ric}^D(X_1,X_1)={\rm Ric}^D(X_2,X_2)=ff''+(f')^2-4,\\
&{\rm Ric}^D(\partial_t,X_1)={\rm Ric}^D(\partial_t,X_2)={\rm Ric}^D(X_1,\partial_t)={\rm Ric}^D(X_2,\partial_t)=0;
\notag\\
&{\rm Ric}^D(X_1,X_2)={\rm Ric}^D(X_2,X_1)=0.\notag
\end{align}
We called that $(D,g^D)$ is Einstein if ${\rm Ric}^D(X,Y)=c_0g^D(X,Y)$ for $X,Y\in\Gamma(D)$.\\

\begin{thm}
$(D,g^D)$ is Einstein with the Einstein constant $c_0$ if and only if\\
\indent ${\rm (1)}~~~c_0=0,~~~f(t)=2t+c_1~~~{\rm or}~~~f(t)=-2t+c_1,$\\
\indent ${\rm (2)}~~~c_0>0,~~~f(t)=-\frac{2}{c_2c_0}e^{\sqrt{\frac{c_0}{2}}t}+c_2e^{-\sqrt{\frac{c_0}{2}}t},$\\
\indent ${\rm (3)}~~~c_0<0,~~~f(t)=c_1{\rm cos}(\sqrt{\frac{-c_0}{2}}t)+c_2{\rm sin}(\sqrt{\frac{-c_0}{2}}t),~~~c_1^2+c^2_2=-\frac{8}{c_0},$\\
where $c_1,c_2$ are constant.
\end{thm}
\begin{proof}
By (5.21), $(D,g^D)$ is Einstein with the Einstein constant $c_0$ if and only if
\begin{align}
&{f''}-\frac{c_0}{2}f=0,\\
&ff''+(f')^2-4=c_0f^2.
\end{align}
\indent If $c_0=0$, by (5.22), then $f=c_2x+c_1$. Using (5.23), then $c_2=2,{\rm or}-2$, so we get case (1).\\
\indent If $c_0>0$, by (5.22), then $f=c_1e^{\sqrt{\frac{c_0}{2}}t}+c_2e^{-\sqrt{\frac{c_0}{2}}t}.$ Using (5.23), then $(f')^2=4+\frac{c_0}{2}f^2$, so $c_1=\frac{-2}{c_2c_0}$ and we get case (2).\\
\indent If $c_0<0$, by (5.22), then $f=c_1{\rm cos}(\sqrt{\frac{-c_0}{2}}t)+c_2{\rm sin}(\sqrt{\frac{-c_0}{2}}t)$. Using $(f')^2=4+\frac{c_0}{2}f^2$, we get $c_1^2+c^2_2=-\frac{8}{c_0}$ and so case (3) holds.
\end{proof}
\vskip 1 true cm
We also get
\begin{align}
s^D:=&{\rm Ric}^D(\partial_t,\partial_t)+{\rm Ric}^D(\widetilde{X}_1,\widetilde{X}_1)+{\rm Ric}^D(\widetilde{X}_2,\widetilde{X}_2)\\
=&4\frac{f''}{f}+2\frac{(f')^2}{f^2}-\frac{8}{f^2}.\notag
\end{align}
\indent Let $U=\partial_t$, then
\begin{align}
\widetilde{\nabla}^D_XY=\nabla^D_XY+g(\partial_t,Y)X-g(X,Y)\partial_t,~~~\widetilde{B}(X,Y)=B(X,Y).
\end{align}
By (5.25), we get
\begin{align}
&\widetilde{\nabla}^D_{\partial_t}\partial_t=0,~~~ \widetilde{\nabla}^D_{\partial_t}X_1=\frac{f'}{f}X_1,~~~ \widetilde{\nabla}^D_{X_1}\partial_t=\frac{f'}{f}X_1+\partial_t,~~~
\widetilde{\nabla}^D_{\partial_t}X_2=\frac{f'}{f}X_2,\\
&~~~ \widetilde{\nabla}^D_{X_2}\partial_t=\frac{f'}{f}X_2+\partial_t,~~~
\widetilde{\nabla}^D_{X_1}X_1=\widetilde{\nabla}^D_{X_2}X_2=-ff'\partial_t-f^2\partial_t,\notag\\
&~~~ \widetilde{\nabla}^D_{X_1}X_2=0,~~~ \widetilde{\nabla}^D_{X_2}X_1=0.\notag
\end{align}
Then
\begin{align}
&\widetilde{R}^D(\partial_t,X_1)\partial_t=\frac{f''}{f}X_1,~~~ \widetilde{R}^D(\partial_t,X_2)\partial_t=\frac{f''}{f}X_2,\\
&
\widetilde{R}^D(\partial_t,X_1)X_1=-(ff''+ff')\partial_t,~~~ \widetilde{R}^D(\partial_t,X_1)X_2=\widetilde{R}^D(\partial_t,X_2)X_1=0,~~~ \notag\\
&
\widetilde{R}^D(\partial_t,X_2)X_2=-(ff''+ff')\partial_t,~~~\widetilde{R}^D(X_1,X_2)\partial_t=\frac{f'}{f}(X_1-X_2),\notag\\
&
\widetilde{R}^D(X_1,X_2)X_1=[(f')^2+ff'-4]X_2+(ff'+f^2)\partial_t,\notag\\
& \widetilde{R}^D(X_1,X_2)X_2=-[(f')^2+ff'-4]X_1-(ff'+f^2)\partial_t,\notag
\end{align}
and
\begin{align}
\widetilde{K}^D(\partial_t\wedge \widetilde{X_1})=-\frac{2f''+f'}{2f},~~~\widetilde{K}^D(\partial_t\wedge \widetilde{X_2})=-\frac{2f''+f'}{2f},~~~\widetilde{K}^D(\widetilde{X_1}\wedge \widetilde{X_2})
=\frac{4-ff'-(f')^2}{f^2},
\end{align}
and
\begin{align}
&\widetilde{{\rm Ric}}^D(\partial_t,\partial_t)=\frac{2f''}{f},~~~ \widetilde{{\rm Ric}}^D(X_1,X_1)={\rm Ric}^D(X_2,X_2)=ff''+2ff'+(f')^2-4,\\
&\widetilde{{\rm Ric}}^D(\partial_t,X_1)=\widetilde{{\rm Ric}}^D(\partial_t,X_2)=0,~~~ \widetilde{{\rm Ric}}^D(X_1,\partial_t)=\widetilde{{\rm Ric}}^D(X_2,\partial_t)=-\frac{f'}{f};
\notag\\
&\widetilde{{\rm Ric}}^D(X_1,X_2)=\widetilde{{\rm Ric}}^D(X_2,X_1)=0.\notag
\end{align}
Then we have

\begin{prop}
$(D,g^D,\widetilde{\nabla}^D)$ is mixed Ricci flat if and only if $f$ is a constant.
\end{prop}
Similarly to (5.24), we have
\begin{align}
\widetilde{s}^D=4\frac{f''}{f}+4\frac{f'}{f}+2\frac{(f')^2}{f^2}-\frac{8}{f^2}.
\end{align}
So when $f$ is a constant, then $\widetilde{s}^D$ is a constant. By (5.29), $(D,g^D,\widetilde{\nabla}^D)$ is not Einstein.\\
\indent By (3.3) and (5.17), we have
\begin{align}
&\widehat{\nabla}^D_{\partial_t}\partial_t=\partial_t,~~~ \widehat{\nabla}^D_{\partial_t}X_1=\frac{f'}{f}X_1,~~~ \widehat{\nabla}^D_{X_1}\partial_t=\frac{f'}{f}X_1+\partial_t,~~~
\widehat{\nabla}^D_{\partial_t}X_2=\frac{f'}{f}X_2,\\
&~~~ \widehat{\nabla}^D_{X_2}\partial_t=\frac{f'}{f}X_2+\partial_t,~~~
\widehat{\nabla}^D_{X_1}X_1=\widehat{\nabla}^D_{X_2}X_2=-ff'\partial_t,\notag\\
&~~~ \widehat{\nabla}^D_{X_1}X_2=0,~~~ \widehat{\nabla}^D_{X_2}X_1=0.\notag
\end{align}
Then
\begin{align}
&\widehat{R}^D(\partial_t,X_1)\partial_t=\frac{f''-f'}{f}X_1,~~~ \widehat{R}^D(\partial_t,X_2)\partial_t=\frac{f''-f'}{f}X_2,\\
&
\widehat{R}^D(\partial_t,X_1)X_1=-(ff''+ff')\partial_t,~~~ \widehat{R}^D(\partial_t,X_1)X_2=\widehat{R}^D(\partial_t,X_2)X_1=0,~~~ \notag\\
&
\widehat{R}^D(\partial_t,X_2)X_2=-(ff''+ff')\partial_t,~~~\widehat{R}^D(X_1,X_2)\partial_t=\frac{f'}{f}(X_1-X_2),\notag\\
&
\widehat{R}^D(X_1,X_2)X_1=[(f')^2-4]X_2+ff'\partial_t,\notag\\
&\widehat{R}^D(X_1,X_2)X_2=-[(f')^2-4]X_1-ff'\partial_t,\notag
\end{align}
and
\begin{align}
&\widehat{{\rm Ric}}^D(\partial_t,\partial_t)=2\frac{f''-f'}{f},~~~ \widehat{{\rm Ric}}^D(X_1,X_1)={\rm Ric}^D(X_2,X_2)=ff''+ff'+(f')^2-4,\\
&\widehat{{\rm Ric}}^D(\partial_t,X_1)=\widehat{{\rm Ric}}^D(\partial_t,X_2)=0,~~~ \widehat{{\rm Ric}}^D(X_1,\partial_t)=\widehat{{\rm Ric}}^D(X_2,\partial_t)=-\frac{f'}{f};
\notag\\
&\widehat{{\rm Ric}}^D(X_1,X_2)=\widehat{{\rm Ric}}^D(X_2,X_1)=0.\notag
\end{align}
So $(D,g^D,\widehat{\nabla}^D)$ is mixed Ricci flat if and only if $f$ is a constant and $(D,g^D,\widehat{\nabla}^D)$ is not Einstein.
Similarly to (5.28) and (5.30), we have
\begin{align}
\widehat{K}^D(\partial_t\wedge \widetilde{X_1})=-\frac{f''}{f},~~~\widehat{K}^D(\partial_t\wedge \widetilde{X_2})=-\frac{f''}{f},~~~\widehat{K}^D(\widetilde{X_1}\wedge \widetilde{X_2})
=\frac{4-(f')^2}{f^2},
\end{align}
\begin{align}
\widehat{s}^D=4\frac{f''}{f}+2\frac{(f')^2}{f^2}-\frac{8}{f^2}.
\end{align}

\vskip 1 true cm
\noindent {\bf Example 3}\\
\indent The Heisenberg group $H_3$ is defined as $\mathbb{R}^3$ with the group operation
\begin{align}
(x,y,z)\cdot(\overline{x},\overline{y},\overline{z})=(x+\overline{x},y+\overline{y},z+\overline{z}+\frac{1}{2}(\overline{x}y-\overline{y}x)).
\end{align}
Let $(H_3,g_{H_3})$ be the Heisenberg group $H_3$ endowed with the Riemannian metric $g$ which is defined by
\begin{align}
g_{H_3}=dx^2+dy^2+(dz+\frac{1}{2}(ydx-xdy))^2.
\end{align}
The following vector fields form a orthonormal frame on $H_3$:
\begin{align}
e_1=\frac{\partial}{\partial x}-\frac{y}{2}\frac{\partial}{\partial z},~~~ e_2=\frac{\partial}{\partial y}+\frac{x}{2}\frac{\partial}{\partial z},~~~ e_3=\frac{\partial}{\partial z}.
\end{align}
These vector fields satisfy the commutation relations
\begin{align}
[e_1,e_2]=e_3,~~~ [e_1,e_3]=0,~~~ [e_2,e_3]=0.
\end{align}
The Levi-Civita connection $\nabla$ of $H_3$ is given by
\begin{align}
&\nabla_{e_j}e_j=0,~~~1\leq j\leq 3,~~~\nabla_{e_1}e_2=\frac{1}{2}e_3,~~~ \nabla_{e_2}e_1=-\frac{1}{2}e_3,\\
& \nabla_{e_1}e_3=\nabla_{e_3}e_1=-\frac{1}{2}e_2,
 ~~~\nabla_{e_2}e_3=\nabla_{e_3}e_2=\frac{1}{2}e_1.\notag
\end{align}
Let $D=span\{e_1,e_2\}$, by (5.40), then $\nabla^D_{e_i}e_j=0,$ $1\leq i,j\leq 2$. Let $U=e_1+e_2+e_3$, then
\begin{align}
&\widetilde{\nabla}^D_{e_1}e_1=-e_2,~~~\widetilde{\nabla}^D_{e_1}e_2=e_1,~~~ \widetilde{\nabla}^D_{e_2}e_1=e_2,~~~
\widetilde{\nabla}^D_{e_2}e_2=-e_1,\\
&
\widetilde{B}(e_1,e_1)=-e_3,~~~ \widetilde{B}(e_2,e_2)=-e_3,~~~ \widetilde{B}(e_1,e_2)=\frac{1}{2}e_3,~~~ \widetilde{B}(e_2,e_1)=-\frac{1}{2}e_3.\notag\\
&\widetilde{R}^D(e_1,e_2)e_1=\widetilde{R}^D(e_1,e_2)e_2=0,\notag
\end{align}
so $(D,g^D,\widetilde{\nabla}^D)$ is flat. Similarly, we have
\begin{align}
&\widehat{\nabla}^D_{e_1}e_1=e_1,~~~\widehat{\nabla}^D_{e_1}e_2=e_1,~~~ \widehat{\nabla}^D_{e_2}e_1=e_2,~~~
\widehat{\nabla}^D_{e_2}e_2=e_2,\\
&\widehat{R}^D(e_1,e_2)e_1=\widehat{R}^D(e_1,e_2)e_2=e_1-e_2.\notag
\end{align}

\vskip 1 true cm
\noindent {\bf Example 4}\\
\indent Let $M=\mathbb{R}\times H_3$ and $D^1=span\{e_1,e_2\}$ and $TH_3=D^1\oplus D^{1,\bot}$. Let $f(t)\in C^{\infty}(\mathbb{R})$ without zero points. Let $\pi_1:
\mathbb{R}\times H_3\rightarrow \mathbb{R};(t,x)\rightarrow t$ and $\pi_2:
\mathbb{R}\times H_3\rightarrow H_3;(t,x)\rightarrow x$. Let
\begin{align}
&g^M_f=\pi^*_1dt^2\oplus f^2\pi^*_2g^{D^1}\oplus \pi^*_2 g^{D^{1,\bot}};\\
&D=\pi^*_1(T\mathbb{R})\oplus \pi^*_2D^1;~~~g^D=\pi^*_1dt^2\oplus f^2\pi^*_2g^{D^1}.\notag
\end{align}
Let $\nabla^f$ be the Levi-Civita connection on $(M,g^M_f)$, then by the Koszul formula and (5.43), we get
\begin{align}
&\nabla^f_{\partial_t}\partial_t=0,~~~ \nabla^f_{\partial_t}e_1=\frac{f'}{f}e_1,~~~ \nabla^f_{e_1}\partial_t=\frac{f'}{f}e_1,~~~
\nabla^f_{\partial_t}e_2=\frac{f'}{f}e_2,\\
&~~~ \nabla^f_{e_2}\partial_t=\frac{f'}{f}e_2,~~~ \nabla^f_{\partial_t}e_3=\nabla^f_{e_3}\partial_t=0,~~~
\nabla^f_{e_1}e_1=\nabla^f_{e_2}e_2=-ff'\partial_t,\notag\\
&~~~ \nabla^f_{e_1}e_2=\frac{1}{2}e_3,~~~ \nabla^f_{e_2}e_1=-\frac{1}{2}e_3,
~~~ \nabla^f_{e_1}e_3=-\frac{e_2}{2f^2},~~~ \nabla^f_{e_3}e_1=-\frac{1}{2f^2}e_2,\notag\\
&~~~
\nabla^f_{e_2}e_3=\frac{e_1}{2f^2},~~~ \nabla^f_{e_3}e_2=\frac{1}{2f^2}e_1,~~~ \nabla^f_{e_3}e_3=0.\notag
\end{align}
So we obtain
\begin{align}
&R^f(\partial_t,e_1)\partial_t=\frac{f''}{f}e_1,~~~ R^f(\partial_t,e_2)\partial_t=\frac{f''}{f}e_2,~~~ R^f(\partial_t,e_3)\partial_t=0,\\
&~~~
R^f(\partial_t,e_1)e_1=-ff''\partial_t,~~~ R^f(\partial_t,e_1)e_2=-\frac{f'}{2f}e_3,~~~ R^f(\partial_t,e_1)e_3=\frac{1}{2}f^{-3}f'e_2,\notag\\
&~~~
 R^f(\partial_t,e_2)e_1=\frac{f'}{2f}e_3,~~~
R^f(\partial_t,e_2)e_2=-ff''\partial_t, R^f(\partial_t,e_2)e_3=-\frac{1}{2}f^{-3}f'e_1, \notag\\
&R^f(\partial_t,e_3)e_1=f^{-3}f'e_2,~~~ R^f(\partial_t,e_3)e_2=-f^{-3}f'e_1,~~~
R^f(\partial_t,e_3)e_3=0,\notag\\
&~~~ R^f(e_1,e_2)\partial_t=\frac{f'}{f}e_3,~~~ R^f(e_1,e_3)\partial_t=\frac{1}{2}f'f^{-3}e_2,~~~ R^f(e_2,e_3)\partial_t=-\frac{1}{2}f^{-3}f'e_1,~~~\notag\\
&
R^f(e_1,e_2)e_1=[(f')^2+\frac{3}{4f^2}]e_2,~~~ R^f(e_1,e_2)e_2=-[(f')^2+\frac{3}{4f^2}]e_1,\notag\\
&~~~ R^f(e_1,e_2)e_3=-f^{-1}f'\partial_t,~~~
R^f(e_1,e_3)e_1=-\frac{1}{4f^2}e_3,~~~ R^f(e_1,e_3)e_2=-\frac{1}{2}f^{-1}f'\partial_t,\notag\\
&~~~ R^f(e_1,e_3)e_3=\frac{1}{4f^4}e_1,~~~ R^f(e_2,e_3)e_1=\frac{1}{2}f^{-1}f'\partial_t,~~~\notag\\
&
R^f(e_2,e_3)e_2=-\frac{1}{4f^2}e_3,~~~ R^f(e_2,e_3)e_3=\frac{1}{4f^4}e_2.\notag
\end{align}
So $M$ is not flat. Then
\begin{align}
&{\rm Ric}^f(\partial_t,\partial_t)=\frac{2f''}{f}, {\rm Ric}^f(e_1,e_1)={\rm Ric}^f(e_2,e_2)=ff''+(f')^2+\frac{1}{2f^2},\\
 &{\rm Ric}^f(X_3,X_3)=-\frac{1}{2}f^{-4},
{\rm Ric}^f(\partial_t,e_1)={\rm Ric}^f(\partial_t,e_2)={\rm Ric}^f(\partial_t,e_3)=0,\notag\\
&{\rm Ric}^f(e_1,e_2)={\rm Ric}^f(e_1,e_3)={\rm Ric}^f(e_2,e_3)=0.\notag
\end{align}
Then $(M,g^M_f)$ is not Einstein and
\begin{align}
{s}^f=4\frac{f''}{f}+2\frac{(f')^2}{f^2}+\frac{1}{2}f^{-4}.
\end{align}

Let $D=span\{\partial_t,e_1,e_2\}$, by (5.44), we have
\begin{align}
&\nabla^D_{\partial_t}\partial_t=0,~~~ \nabla^D_{\partial_t}e_1=\frac{f'}{f}e_1,~~~ \nabla^D_{e_1}\partial_t=\frac{f'}{f}e_1,~~~
\nabla^D_{\partial_t}e_2=\frac{f'}{f}e_2,\\
&~~~ \nabla^D_{e_2}\partial_t=\frac{f'}{f}e_2,~~~
\nabla^D_{e_1}e_1=\nabla^D_{e_2}e_2=-ff'\partial_t,\notag\\
&~~~ \nabla^D_{e_1}e_2=0,~~~ \nabla^D_{e_2}e_1=0.\notag
\end{align}
We can compute the curvature tensor on $D$ as follows:
\begin{align}
&R^D(\partial_t,e_1)\partial_t=\frac{f''}{f}e_1,~~~ R^D(\partial_t,e_2)\partial_t=\frac{f''}{f}e_2,\\
&
R^D(\partial_t,e_1)e_1=-ff''\partial_t,~~~ R^D(\partial_t,e_1)e_2=R^D(\partial_t,e_2)e_1=0,~~~ \notag\\
&
R^D(\partial_t,e_2)e_2=-ff''\partial_t,~~~ R^D(e_1,e_2)\partial_t=0,\notag\\
&
R^D(e_1,e_2)e_1=(f')^2e_2,~~~ R^D(e_1,e_2)e_2=-(f')^2e_1,\notag
\end{align}
and we have
\begin{align}
K^D(\partial_t\wedge \widetilde{e_1})=-\frac{f''}{f},~~~K^D(\partial_t\wedge \widetilde{e_2})=-\frac{f''}{f},~~~K^D(\widetilde{e_1}\wedge \widetilde{e_2})=\frac{(f')^2}{f^2}.
\end{align}
So we have
\begin{align}
&{\rm Ric}^D(\partial_t,\partial_t)=\frac{2f''}{f}, {\rm Ric}^D(e_1,e_1)={\rm Ric}^D(e_2,e_2)=ff''+(f')^2,\\
&{\rm Ric}^D(\partial_t,e_1)={\rm Ric}^D(\partial_t,e_2)={\rm Ric}^D(e_1,\partial_t)={\rm Ric}^D(e_2,\partial_t)=0;
\notag\\
&{\rm Ric}^D(e_1,e_2)={\rm Ric}^D(e_2,e_1)=0.\notag
\end{align}

\begin{thm}
$(D,g^D)$ is Einstein with the Einstein constant $c_0$ if and only if\\
\indent ${\rm (1)}~~~c_0=0,~~~f(t)=c_1,$\\
\indent ${\rm (2)}~~~c_0>0,~~~f(t)=c_1e^{\sqrt{\frac{c_0}{2}}t}~~{\rm or}~~f(t)=c_2e^{-\sqrt{\frac{c_0}{2}}t},$\\
where $c_1,c_2$ are constant.
\end{thm}
\begin{proof}
By (5.51), $(D,g^D)$ is Einstein with the Einstein constant $c_0$ if and only if
\begin{align}
&{f''}-\frac{c_0}{2}f=0,\\
&ff''+(f')^2=c_0f^2.
\end{align}
\indent If $c_0=0$, by (5.52), then $f=c_2x+c_1$. Using (5.53), then $c_2=0$, so we get case (1).\\
\indent If $c_0>0$, by (5.52), then $f=c_1e^{\sqrt{\frac{c_0}{2}}t}+c_2e^{-\sqrt{\frac{c_0}{2}}t}.$ Using (5.53), then $(f')^2=\frac{c_0}{2}f^2$, so $c_1=0$ or $c_2=0$, and we get case (2).\\
\indent If $c_0<0$, by (5.52), then $f=c_1{\rm cos}(\sqrt{\frac{-c_0}{2}}t)+c_2{\rm sin}(\sqrt{\frac{-c_0}{2}}t)$. Using $(f')^2=\frac{c_0}{2}f^2$, we get $c_1=c_2=0$. But $f\neq 0$, so in this case
there is no solution.
\end{proof}

\begin{thm}
$(D,g^D)$ is a distribution with constant scalar curvature $\lambda_0$ if and only if\\
\indent ${\rm (1)}~~~\lambda_0=0,~~~f(t)=(c_2t+c_1)^{\frac{2}{3}},$\\
\indent ${\rm (2)}~~~\lambda_0>0,~~~f(t)=(c_1e^{\sqrt{\frac{3\lambda_0}{8}}t}+c_2e^{-\sqrt{\frac{3\lambda_0}{8}}t})^{\frac{2}{3}},$\\
\indent ${\rm (3)}~~~\lambda_0<0,~~~f(t)=(c_1{\rm cos}(\sqrt{-\frac{3\lambda_0}{8}}t)+c_2{\rm sin}(\sqrt{-\frac{3\lambda_0}{8}}t))^{\frac{2}{3}},$\\
where $c_1,c_2$ are constant.
\end{thm}
\begin{proof}
By (5.51), we have
\begin{align}
{s}^D=4\frac{f''}{f}+2\frac{(f')^2}{f^2}=\lambda_0.
\end{align}
Let $f(t)=w(t)^{\frac{2}{3}}$ and by (5.54), we get $w''(t)-\frac{3}{8}\lambda_0w(t)=0$. By the elementary methods for ordinary differential equations we prove the above theorem.
\end{proof}

\indent Let $U=\partial_t$,
By (5.48), we get
\begin{align}
&\widetilde{\nabla}^D_{\partial_t}\partial_t=0,~~~ \widetilde{\nabla}^D_{\partial_t}e_1=\frac{f'}{f}e_1,~~~ \widetilde{\nabla}^D_{e_1}\partial_t=\frac{f'}{f}e_1+\partial_t,~~~
\widetilde{\nabla}^D_{\partial_t}e_2=\frac{f'}{f}e_2,\\
&~~~ \widetilde{\nabla}^D_{e_2}\partial_t=\frac{f'}{f}e_2+\partial_t,~~~
\widetilde{\nabla}^D_{e_1}e_1=\widetilde{\nabla}^D_{e_2}e_2=-ff'\partial_t-f^2\partial_t,\notag\\
&~~~ \widetilde{\nabla}^D_{e_1}e_2=0,~~~ \widetilde{\nabla}^D_{e_2}e_1=0.\notag
\end{align}
Then
\begin{align}
&\widetilde{R}^D(\partial_t,e_1)\partial_t=\frac{f''}{f}e_1,~~~ \widetilde{R}^D(\partial_t,e_2)\partial_t=\frac{f''}{f}e_2,\\
&
\widetilde{R}^D(\partial_t,e_1)e_1=-(ff''+ff')\partial_t,~~~ \widetilde{R}^D(\partial_t,e_1)e_2=\widetilde{R}^D(\partial_t,e_2)e_1=0,~~~ \notag\\
&
\widetilde{R}^D(\partial_t,e_2)e_2=-(ff''+ff')\partial_t,~~~\widetilde{R}^D(e_1,e_2)\partial_t=\frac{f'}{f}(e_1-e_2),\notag\\
&
\widetilde{R}^D(e_1,e_2)e_1=[(f')^2+ff']e_2+(ff'+f^2)\partial_t,\notag\\
& \widetilde{R}^D(e_1,e_2)e_2=-[(f')^2+ff']e_1-(ff'+f^2)\partial_t,\notag
\end{align}
and
\begin{align}
&\widetilde{{\rm Ric}}^D(\partial_t,\partial_t)=\frac{2f''}{f},~~~ \widetilde{{\rm Ric}}^D(e_1,e_1)={\rm Ric}^D(e_2,e_2)=ff''+2ff'+(f')^2,\\
&\widetilde{{\rm Ric}}^D(\partial_t,e_1)=\widetilde{{\rm Ric}}^D(\partial_t,e_2)=0,~~~ \widetilde{{\rm Ric}}^D(e_1,\partial_t)=\widetilde{{\rm Ric}}^D(e_2,\partial_t)=-\frac{f'}{f};
\notag\\
&\widetilde{{\rm Ric}}^D(e_1,e_2)=\widetilde{{\rm Ric}}^D(e_2,e_1)=0.\notag
\end{align}
Then $(D,g^D,\widetilde{\nabla}^D)$ is mixed Ricci flat if and only if $f$ is a constant.

\begin{thm}
$(D,g^D,\widetilde{\nabla}^D)$ is a distribution with constant scalar curvature $\lambda_0$ for $U=\partial_t$ if and only if\\
\indent ${\rm (1)}~~~\lambda_0=-\frac{2}{3},~~~f(t)=(c_1e^{-\frac{1}{2}t}+c_2te^{-\frac{1}{2}t})^{\frac{2}{3}},$\\
\indent ${\rm (2)}~~~\lambda_0>-\frac{2}{3},~~~f(t)=(c_1e^{\frac{-1+\sqrt{1+\frac{3}{2}\lambda_0}}{2}t}+c_2e^{\frac{-1-\sqrt{1+\frac{3}{2}\lambda_0}}{2}t})^{\frac{2}{3}}
,$\\
\indent ${\rm (3)}~~~\lambda_0<-\frac{2}{3},~~~f(t)=(c_1e^{-\frac{1}{2}t}{\rm cos}(\frac{\sqrt{-(1+\frac{3}{2}\lambda_0)}}{2}t)+c_2e^{-\frac{1}{2}t}{\rm sin}(\frac{\sqrt{-(1+\frac{3}{2}\lambda_0)}}{2}t))^{\frac{2}{3}},$\\
where $c_1,c_2$ are constant.
\end{thm}
\begin{proof}
By (5.57), we have
\begin{align}
\widetilde{s}^D=4\frac{f''}{f}+4\frac{f'}{f}+2\frac{(f')^2}{f^2}=\lambda_0.
\end{align}
Let $f(t)=w(t)^{\frac{2}{3}}$ and by (5.58), we get $w''(t)+w'(t)-\frac{3}{8}\lambda_0w(t)=0$. By the elementary methods for ordinary differential equations we prove the above theorem.
\end{proof}
\indent By (5.48), we have
\begin{align}
&\widehat{\nabla}^D_{\partial_t}\partial_t=\partial_t,~~~ \widehat{\nabla}^D_{\partial_t}e_1=\frac{f'}{f}e_1,~~~ \widehat{\nabla}^D_{e_1}\partial_t=\frac{f'}{f}e_1+\partial_t,~~~
\widehat{\nabla}^D_{\partial_t}e_2=\frac{f'}{f}e_2,\\
&~~~ \widehat{\nabla}^D_{e_2}\partial_t=\frac{f'}{f}e_2+\partial_t,~~~
\widehat{\nabla}^D_{e_1}e_1=\widehat{\nabla}^D_{e_2}e_2=-ff'\partial_t,\notag\\
&~~~ \widehat{\nabla}^D_{e_1}e_2=0,~~~ \widehat{\nabla}^D_{e_2}e_1=0.\notag
\end{align}
Then
\begin{align}
&\widehat{R}^D(\partial_t,e_1)\partial_t=\frac{f''-f'}{f}e_1,~~~ \widehat{R}^D(\partial_t,e_2)\partial_t=\frac{f''-f'}{f}e_2,\\
&
\widehat{R}^D(\partial_t,e_1)e_1=-(ff''+ff')\partial_t,~~~ \widehat{R}^D(\partial_t,e_1)e_2=\widehat{R}^D(\partial_t,e_2)e_1=0,~~~ \notag\\
&
\widehat{R}^D(\partial_t,e_2)e_2=-(ff''+ff')\partial_t,~~~\widehat{R}^D(e_1,e_2)\partial_t=\frac{f'}{f}(e_1-e_2),\notag\\
&
\widehat{R}^D(e_1,e_2)e_1=(f')^2e_2+ff'\partial_t,\notag\\
&\widehat{R}^D(e_1,e_2)e_2=-(f')^2e_1-ff'\partial_t,\notag
\end{align}
and
\begin{align}
&\widehat{{\rm Ric}}^D(\partial_t,\partial_t)=2\frac{f''-f'}{f},~~~ \widehat{{\rm Ric}}^D(e_1,e_1)={\rm Ric}^D(e_2,e_2)=ff''+ff'+(f')^2,\\
&\widehat{{\rm Ric}}^D(\partial_t,e_1)=\widehat{{\rm Ric}}^D(\partial_t,e_2)=0,~~~ \widehat{{\rm Ric}}^D(e_1,\partial_t)=\widehat{{\rm Ric}}^D(e_2,\partial_t)=-\frac{f'}{f};
\notag\\
&\widehat{{\rm Ric}}^D(e_1,e_2)=\widehat{{\rm Ric}}^D(e_2,e_1)=0.\notag
\end{align}
So $(D,g^D,\widehat{\nabla}^D)$ is mixed Ricci flat if and only if $f$ is a constant. By (5.61), we get
\begin{align}
\widehat{s}^D=4\frac{f''}{f}+2\frac{(f')^2}{f^2}=s^D.
\end{align}
By Theorem 5,4, we have
\begin{thm}
$(D,g^D,\widehat{\nabla}^D)$ is a distribution with constant scalar curvature $\lambda_0$ for $U=\partial_t$ if and only if\\
\indent ${\rm (1)}~~~\lambda_0=0,~~~f(t)=(c_2t+c_1)^{\frac{2}{3}},$\\
\indent ${\rm (2)}~~~\lambda_0>0,~~~f(t)=(c_1e^{\sqrt{\frac{3\lambda_0}{8}}t}+c_2e^{-\sqrt{\frac{3\lambda_0}{8}}t})^{\frac{2}{3}},$\\
\indent ${\rm (3)}~~~\lambda_0<0,~~~f(t)=(c_1{\rm cos}(\sqrt{-\frac{3\lambda_0}{8}}t)+c_2{\rm sin}(\sqrt{-\frac{3\lambda_0}{8}}t))^{\frac{2}{3}},$\\
where $c_1,c_2$ are constant.
\end{thm}

\vskip 1 true cm

\section{Acknowledgements}

The author was supported in part by  NSFC No.11771070.

\vskip 1 true cm

%-----------------------------------------------------------------------------
%-----------------------------------------------------------------------------

\bigskip
\bigskip

\noindent {\footnotesize {\it Y. Wang} \\
{School of Mathematics and Statistics, Northeast Normal University, Changchun 130024, China}\\
{Email: wangy581@nenu.edu.cn}

\end{document}